\documentclass[12pt,a4paper]{article}

\global

\usepackage{helvet}         % selects Helvetica as sans-serif font
\usepackage{courier}        % selects Courier as typewriter font
\usepackage{type1cm}        % activate if the above 3 fonts are
\usepackage{graphicx}        % standard LaTeX graphics tool
\usepackage{multicol}        % used for the two-column index
\usepackage[bottom]{footmisc}% places footnotes at page bottom

\usepackage[utf8]{inputenc}
\usepackage[T1]{fontenc}
\usepackage{amsmath}
\usepackage{amssymb}
\usepackage[english]{babel}

\usepackage{amsfonts}

\usepackage{fouridx}
\usepackage{amsthm}
\usepackage{cite}
\usepackage{varioref}
\usepackage{cleveref}

\usepackage{authblk}

\usepackage{enumitem}  % http://ctan.org/pkg/enumitem
\usepackage{calc}% http://ctan.org/pkg/calc
\setlist{labelindent=1pt,itemsep=.5em}
\setlist[itemize]{leftmargin=1.2cm}
\setlist[enumerate]{itemindent=0em,leftmargin=1.2cm}
\setlist[enumerate,1]{label={\upshape(\roman*)}}

\usepackage{multicol}

\allowdisplaybreaks[3]

\newtheorem{theorem}{Theorem}[section]
\newtheorem{cor}[theorem]{Corollary}
\newtheorem{thm}[theorem]{Theorem}
\newtheorem{lem}[theorem]{Lemma}
\newtheorem{prop}[theorem]{Proposition}
\newtheorem{exmpl}{Example}[section]
\newtheorem{defn}[theorem]{Definition}
\newtheorem{remq}[theorem]{Remark}
\newcommand{\N}{\mathbb{N}}
\newcommand{\C}{\mathbb{C}}

\newcommand{\K}{\mathbb{K}}

\makeatletter
\newcommand{\subjclass}[2][2020]{%
  \let\@oldtitle\@title%
  \gdef\@title{\@oldtitle\footnotetext{#1 \emph{Mathematics subject classification}: #2}}%
}
\newcommand{\keywords}[1]{%
  \let\@@oldtitle\@title%
  \gdef\@title{\@@oldtitle\footnotetext{\emph{Keywords}: #1}}%
}
\makeatother

\title{On $(\lambda,\mu,\gamma)$-derivations of BiHom-Lie algebras}
\author{Nejib Saadaoui$^{1}$, Sergei Silvestrov$^{2}$ \authorcr
\small{$^{1}$Université de Gabès, Institut Supérieur de l'Informatique Medenine, \authorcr
Rue Djerba km3, B.P 283 Medenine 4100, Tunisie.\authorcr
e-mail: najibsaadaoui@yahoo.fr
\authorcr
$^{2}$ Division of Mathematics and Physics, School of Education, Culture and Communication, \authorcr
M{\"a}lardalen University, Box 883, 72123 V{\"a}ster{\aa}s, Sweden. \authorcr
e-mail: sergei.silvestrov@mdh.se}}

\subjclass[2020]{17B61, 17D30}
\keywords{BiHom-Lie algebra, BiHom-Lie derivation, derivation, centroid}

\date{}

\begin{document}

\maketitle

%\footnote[0]{{\it Corresponding authors}: Nejib Saadaou, Sergei Silvestrov}

\abstract{In this paper, we generalize the results  about generalized
derivations of Lie algebras to the case of BiHom-Lie algebras. In particular  we give the classification of generalized derivations of Heisenberg BiHom-Lie algebras. The definition of the generalized derivation depends on some parameters $ (\lambda,\mu,\gamma)\in \C^{3}. $
In particular  for  $(\lambda,\mu,\gamma)=(1,1,1)$, we obtain classical concept
of derivation of BiHom-Lie algebra  and for $(\lambda,\mu,\gamma)=(1,1,0)  $ we obtain the centroid of
BiHom-Lie algebra. We give classifications of $ 2 $-dimensional
BiHom-Lie algebra, centroides and derivations of  $ 2 $-dimensional
BiHom-Lie algebras.}

\section{Introduction}
The investigations of various quantum deformations or $q$-deformations of Lie algebras began a period of rapid expansion in 1980's stimulated by introduction of quantum groups motivated by applications to the quantum Yang-Baxter equation, quantum inverse scattering methods and constructions of the quantum deformations of universal enveloping algebras of semi-simple Lie algebras. Various $q$-deformed Lie algebras have appeared in physical contexts such as string theory, vertex models in conformal field theory, quantum mechanics and quantum field theory in the context of deformations of infinite-dimensional algebras, primarily the Heisenberg algebras, oscillator algebras and Witt and Virasoro algebras. In \cite{AizawaSaito,ChaiElinPop,ChaiIsLukPopPresn,ChaiKuLuk,ChaiPopPres,CurtrZachos1,DamKu,DaskaloyannisGendefVir,Hu,Kassel92,LiuKQuantumCentExt,LiuKQCharQuantWittAlg,LiuKQPhDthesis},
it was in particular discovered that in these $q$-deformations of Witt and Visaroro algebras and some related algebras, some interesting $q$-deformations of Jacobi identities, extending Jacobi identity for Lie algebras, are satisfied. This has been one of the initial motivations for the development of general quasi-deformations and discretizations of Lie algebras of vector fields using more general $\sigma$-derivations (twisted derivations) in \cite{HartwigLarssonSilvestrov:defLiealgsderiv}.

Hom-Lie algebras and more general quasi-Hom-Lie algebras were introduced first by Hartwig, Larsson and Silvestrov \cite{HartwigLarssonSilvestrov:defLiealgsderiv}, where the general quasi-deformations and discretizations of Lie algebras of vector fields using more general $\sigma$-derivations (twisted derivations) and a general method for construction of deformations of Witt and Virasoro type algebras based on twisted derivations have been developed, initially motivated by the $q$-deformed Jacobi identities observed for the $q$-deformed algebras in physics, along with $q$-deformed versions of homological algebra and discrete modifications of differential calculi. Hom-Lie algebras, Hom-Lie superalgebras, Hom-Lie color algebras and more general quasi-Lie algebras and color quasi-Lie algebras where introduced first in \cite{LarssonSilv2005:QuasiLieAlg,LarssonSilv:GradedquasiLiealg,SigSilv:CzechJP2006:GradedquasiLiealgWitt}. Quasi-Lie algebras and color quasi-Lie algebras encompass within the same algebraic framework the quasi-deformations and discretizations of Lie algebras of vector fields by $\sigma$-derivations obeying twisted Leibniz rule, and the well-known generalizations of Lie algebras such as color Lie algebras, the natural generalizations of Lie algebras and Lie superalgebras. In quasi-Lie algebras, the skew-symmetry and the Jacobi identity are twisted by deforming twisting linear maps, with the Jacobi identity in quasi-Lie and quasi-Hom-Lie algebras in general containing six twisted triple bracket terms. In Hom-Lie algebras, the bilinear product satisfies the non-twisted skew-symmetry property as in Lie algebras, and the Hom-Lie algebras Jacobi identity has three terms twisted by a single linear map, reducing to the Lie algebras Jacobi identity when the twisting linear map is the identity map. Hom-Lie admissible algebras have been considered first in \cite{ms:homstructure}, where in particular the Hom-associative algebras have been introduced and shown to be Hom-Lie admissible, that is leading to Hom-Lie algebras using commutator map as new product, and in this sense constituting a natural generalization of associative algebras as Lie admissible algebras. Since the pioneering works \cite{HartwigLarssonSilvestrov:defLiealgsderiv,LarssonSilvJA2005:QuasiHomLieCentExt2cocyid,LarssonSilv:GradedquasiLiealg,LarssonSilv2005:QuasiLieAlg,LarssonSilv:QuasidefSl2,ms:homstructure}, Hom-algebra structures expanded into a popular area with increasing number of publications in various directions. Hom-algebra structures of a given type include their classical counterparts and open broad possibilities for deformations, Hom-algebra extensions of cohomological structures and representations, formal deformations of Hom-associative and Hom-Lie algebras, Hom-Lie admissible Hom-coalgebras, Hom-coalgebras, Hom-Hopf algebras \cite{AmmarEjbehiMakhlouf:homdeformation,BenMakh:Hombiliform,ElchingerLundMakhSilv:BracktausigmaderivWittVir,LarssonSilvJA2005:QuasiHomLieCentExt2cocyid,LarssonSilvestrovGLTMPBSpr2009:GenNComplTwistDer,MakhSil:HomHopf,MakhSilv:HomAlgHomCoalg,MakhSilv:HomDeform,
Sheng:homrep,ShengBai2014:homLiebialg,Yau:HomolHom,Yau:EnvLieAlg}.
Hom-Lie algebras, Hom-Lie superalgebras and color Hom-Lie algebras and their $n$-ary generalizations have been further investigated in various aspects for example in \cite{AbdaouiAmmarMakhloufCohhomLiecolalg2015,AbramovSilvestrov:3homLiealgsigmaderivINvol,AmmarEjbehiMakhlouf:homdeformation,AmmarMabroukMakhloufCohomnaryHNLalg2011,AmmarMakhloufHomLieSupAlg2010,AmmarMakhloufSaadaoui2013:CohlgHomLiesupqdefWittSup,AmmarMakhloufSilv:TernaryqVirasoroHomNambuLie,ArmakanFarhangdoost:IJGMMP,ArmakanSilv:envelalgcertaintypescolorHomLie,ArmakanSilvFarh:envelopalgcolhomLiealg,ArmakanSilvFarh:exthomLiecoloralg,ArmakanSilv:NondegKillingformsHomLiesuperalg,
akms:ternary,ams:ternary,ArnlindMakhloufSilvnaryHomLieNambuJMP2011,Bakayoko2014:ModulescolorHomPoisson,BakayokoDialo2015:genHomalgebrastr,BakyokoSilvestrov:Homleftsymmetriccolordialgebras,BakyokoSilvestrov:MultiplicnHomLiecoloralg,BakayokoToure2019:genHomalgebrastr,
BenHassineMabroukNcib:ConstrMultiplicnaryhomNambualg,BenHassineChtiouiMabroukNcib:Strcohom3LieRinehartsuperalg,BenMakh:Hombiliform,CaoChen2012:SplitregularhomLiecoloralg,GuanChenSun:HomLieSuperalgebras,KitouniMakhloufSilvestrov,kms:solvnilpnhomlie2020,LarssonSigSilvJGLTA2008:QuasiLiedefFttN,MabroukNcibSilvestrov2020:GenDerRotaBaxterOpsnaryHomNambuSuperalgs,
ms:homstructure,MakhSilv:HomDeform,MakhSil:HomHopf,MakhSilv:HomAlgHomCoalg,Makhlouf2010:ParadigmnonassHomalgHomsuper,MandalMishra:HomGerstenhaberHomLiealgebroids,MishraSilvestrov:SpringerAAS2020HomGerstenhalgsHomLiealgds,RichardSilvJA2008:quasiLiesigderCtpm1,RichardSilvestrovGLTMPBSpr2009:QuasiLieHomLiesigmaderiv,
SigSilv:CzechJP2006:GradedquasiLiealgWitt,Sheng:homrep,ShengBai2014:homLiebialg,ShengChen2013:HomLie2algebras,ShengXiong:LMLA2015:OnHomLiealg,SigSilv:GLTbdSpringer2009,SilvestrovParadigmQLieQhomLie2007,
Yau2009:HomYangBaxterHomLiequasitring,Yau:EnvLieAlg,Yau:HomolHom,Yau:HomBial,Yuan2012:HomLiecoloralgstr,ZhouChenMa:GenDerHomLiesuper}.
In \cite{CanepeelGoyaverts:MonoidalHomHopfalgebras}, Hom-algebras has been considered from a category theory point of view, constructing a category on which algebras would be Hom-algebras. A generalization of this approach led to the discovery of BiHom-algebras in \cite{GrazMakhlMeniniPanaite:bihom}, called BiHom-algebras because the defining identities are twisted by two morphisms instead of only one for Hom-algebras.
BiHom-Frobenius algebras and double constructions have been investigated in \cite{HounkonnouHoundedjiSilvestrov:DoubleconstrbiHomFrobalg}.

Derivations and generalized derivations of different algebraic structures are an important subject of study in algebra and diverse areas. They appear in many fields of Mathematics and Physics. In particular, they appear in representation theory and cohomology theory among other areas. They have various applications relating algebra to geometry and allow the construction of new algebraic structures.  There are many generalizations of derivations, for example, Leibniz derivations and $\delta$-derivations of prime Lie and Malcev algebras and related $n$-ary algebras structures \cite{DalTakh,Filippov:nLie,Filippov1998:deltaderivLiealg,Filippov1999:deltaderivprimeLiealg,Filippov2000:deltaderivprimealternMalcevalg,KaygorodovPopov2014:Altalgadmderiv}.
The properties and structure of generalized derivations algebras of a Lie algebra and their subalgebras and quasi-derivation algebras were systematically studied in \cite{LegerLuks:GenDerivLiealg}, where it was proved for example that the quasi-derivation algebra of a Lie algebra can be embedded into the derivation algebra of a larger Lie algebra. Derivations and generalized derivations of $n$-ary algebras were considered in \cite{PojidaevSaraiva:Derivternmalcevalg,Williams:NilpnLiealg} and it was demonstrated substantial differences in structures and properties of derivations on Lie algebras and on $n$-ary Lie algebras for $n>2$. Generalized derivations of Lie superalgebras and Hom-Leibniz algebras have been considered in \cite{ZhangZhang:GenDerLieSuperalg,ZhouZhaoZhang:GenDerHomLeibnizalg}.
Generalized derivations of Lie color algebras and $n$-ary (color) algebras have been studied in \cite{ChenMaNi:GenDerLieColAlg, Kaygorodov2012:deltaDerivnaryalg,Kaygorodov2011:nplus1Aryderivsimplenaryalg,Kaygorodov2014:nplus1AryderivsemisimpleFilipovalg, KaygorodovPopov2016:GeneralizedderivcolornLiealg}. Generalized derivations of Lie triple systems have been considered in \cite{ZhouChenMa:GenDerLieTripSyst}. Generalized derivations of various kinds can be viewed as a generalization of $\delta$-derivation. Quasi-Hom-Lie and Hom-Lie structures for $\sigma$-derivations and $(\sigma,\tau)$-derivations have been considered in \cite{ElchingerLundMakhSilv:BracktausigmaderivWittVir,HartwigLarssonSilvestrov:defLiealgsderiv,LarssonSilv:QuasidefSl2,RichardSilvJA2008:quasiLiesigderCtpm1,RichardSilvestrovGLTMPBSpr2009:QuasiLieHomLiesigmaderiv}.
Graded $q$-differential algebra and applications to semi-commutative Galois Extensions and Reduced Quantum Plane and $q$-connection was studied in \cite{AbramovJNMPh2006:gradedqdifalg,AbramovGLTMPBSpr2009:GradedqDiffAlgqConnect,AbramovRaknuz2016EngMathSpr:SemicomGaloisExtRedQPl}.
Generalized $N$-complexes coming from twisted derivations where considered in \cite{LarssonSilvestrovGLTMPBSpr2009:GenNComplTwistDer}.

Generalizations of derivations in connection with extensions and enveloping algebras of Hom-Lie color algebras and Hom-Lie superalgebras have been considered in \cite{ArmakanSilvFarh:envelopalgcolhomLiealg,ArmakanSilvFarh:exthomLiecoloralg,BakyokoSilvestrov:MultiplicnHomLiecoloralg,GuanChenSun:HomLieSuperalgebras}. Generalized derivations of multiplicative $n$-ary Hom-$\Omega$ color algebras have been studied in \cite{BeitesKaygorodovPopov}. Derivations, $L$-modules, $L$-comodules and Hom-Lie quasi-bialgebras have been considered in \cite{Bakayoko:LmodcomodhomLiequasibialg,Bakayoko:LaplacehomLiequasibialg}. In \cite{kms:narygenBiHomLieBiHomassalgebras2020}, constructions of $n$-ary generalizations of BiHom-Lie algebras and BiHom-associative algebras have been investigated. Generalized Derivations of $n$-BiHom-Lie algebras have been studied in \cite{BenAbdeljElhamdKaygorMakhl201920GenDernBiHomLiealg}. Color Hom-algebra structures associated to Rota-Baxter operators have been considered in context of Hom-dendriform color algebras in \cite{BakyokoSilvestrov:Homleftsymmetriccolordialgebras}. Rota-Baxter bisystems and covariant bialgebras, Rota-Baxter cosystems, coquasitriangular mixed bialgebras, coassociative Yang-Baxter pairs, coassociative Yang-Baxter equation and generalizations of Rota-Baxter systems and algebras, curved $\mathcal{O}$-operator systems and their connections with (tri)dendriform systems and pre-Lie algebras have been considered in \cite{MaMakhSil:CurvedOoperatorSyst,MaMakhSil:RotaBaxbisyscovbialg,MaMakhSil:RotaBaxCosyCoquasitriMixBial}.
Generalisations of derivations are important for Hom-Gerstenhaber algebras, Hom-Lie algebroids and Hom-Lie-Rinehart algebras and Hom-Poisson homology
\cite{MishraSilvestrov:SpringerAAS2020HomGerstenhalgsHomLiealgds}.

It is well known that for a derivation $ d $ of  Lie algebra $ L $,  is just an
endomorphisms  on $ L $ such that
\begin{equation}\label{LieDerivation}
 d\left([x,y] \right)=\left[ d(x),y\right] +\left[x,d(y) \right]
\end{equation}
 for all $ x,\, y \in L. $
There were several non–equivalent ways generalizing this definition, for example:
%have recently been studied
%approaches to the study of generalized derivations on Lie algebras:
\begin{enumerate}[label=\upshape{\arabic*)},left=0pt]	
\item  The mapping $ d\in End(L) $ is called a generalized derivation of
	 $ L $ if there exist elements $ d',\, d''\in End(L) $ such that,
\begin{equation} \label{genlzdDerivation}
\left[ d(x),y\right] +\left[x,d'(y) \right]=d''\left([x,y] \right)
\end{equation}
for all $ x,\, y \in L $, and  we call $ d\in End(L) $ a quasiderivation of $ L $ if there exists $ d'\in End(L) $
 such that
\begin{equation} \label{quasiDerivation}
\left[ d(x),y\right] +\left[x,d(y) \right]=d'\left([x,y] \right).
\end{equation}
The centroid of $ L $ denoted as $ \varGamma(L) $ is defined by
\begin{equation}
  \varGamma(L)=\{d\in End(L)\mid d\left([x,y] \right)=\left[ d(x),y\right] =\left[x,d(y) \right],\, \forall x,y\in L   \}  .
\end{equation}
(see for example \cite{GeorgeEugen:GenderivLiealg}).	
	\item %Let $ (A,\mu) $ be an algebra over a unitai commutative associative ring $ \K $.
	 Given an arbitrary $ \delta\in \K $, a $ \delta $-derivation of a Lie algebra $  L $ is defined to be a $ \K $-linear mapping $ d\colon L\to L $
	satisfying the identity
\begin{equation}	
	 d\left([ x,y]\right) =\delta\left[  d(x),y \right]  +\delta\left[ x,d(y) \right]
\end{equation}	
(see for example \cite{Filippov1998:deltaderivLiealg}). Observe that,  any  linear mapping in the centroid $  \varGamma(L) $ is a $ \frac{1}{2} $-derivation of $ L $.
	\item We call a linear operator $ d\in End(L) $  an
	$ (\alpha,\beta,\gamma) $-derivation of $ L $ if there exist $\alpha,\beta,\gamma\in \K  $ such that for all $ x,y\in L $
\begin{equation}	
\alpha\, d\left([ x,y]\right) =\beta\left[  d(x),y \right]  +\gamma\left[ x,d(y) \right].
\end{equation}
(See for example \cite{NovotnyHrivnak:abgderivLiealinvfncs}). Observe that,  any  linear mapping in the centroid $  \varGamma(L) $ is a $ (1,1,0)$-derivation of $ L $.
\end{enumerate}
In \cite{Sheng:homrep}, the notion of $ \alpha^{k} $-derivation of  Hom-Lie
algebra, a generalization of derivation of  Lie algebras
\eqref{LieDerivation}, is considered. In \cite{ZhouNiuChen:GhomDerivation}
the authors extend the definition of type \eqref{LieDerivation} of a generalized derivation of Lie algebras to Hom-Lie algebras. The definition of type \eqref{LieDerivation} is extended to the BiHom-Lie case in \cite{belhsine}.
In this article, we aim to discuss the version \eqref{quasiDerivation} of  generalized derivations of BiHom-Lie algebras.

The paper is organized as follows. In Section \ref{sec:defprelimres}, we recall some basic definitions  and facts needed
later for considerations and results in this article. In Section \ref{sec:genderivBiHomLiealg},  we introduce $ (\lambda,\mu,\gamma) $-$ \alpha^k\beta^l $-derivations and show their pertinent properties. Also, we classify the possible values of $ \lambda,\mu,\gamma \in\C $ for a space
 $ Der^{\lambda,\mu,\gamma }_{\alpha^{k}\beta^{l}}(G) $ of
$ (\lambda,\mu,\gamma)  $-$ \alpha^k\beta^l $-derivations of regular BiHom-Lie algebra $ G $.
The previous classification is applied to Heisenberg BiHom-Lie algebra case.
Next, we analyze each one
of the following cases: $ Der^{\delta,0,0}_{\alpha^{k}\beta^{l}}(G) $ with $ \delta\in \{0,1\} $,
$ Der^{\delta,1,0}_{\alpha^{k}\beta^{l}}(G) $,
$ Der^{\delta,1,1}_{\alpha^{k}\beta^{l}}(G) $,
$ Der^{1,1,-1}_{\alpha^{k}\beta^{l}}(G) $,
$ Der^{0,1,-1}_{\alpha^{k}\beta^{l}}(G) $.
In Section \ref{sec:classmultiplic2dimBiHomLiealg}, we give a method to determine whether two different  $ 2 $-dimensional multiplicative BiHom-Lie algebras are isomorphic or not, and then we obtain a complete classification of   $ 2 $-dimensional multiplicative  BiHom-Lie algebras
up to isomorphism.   In Section \ref{sec:centrderiv2dimmultipBiHomLiealg} we deal with the
problem of description of centroids and derivations of $ 2 $-dimensional BiHom Lie  algebras. Here we provide
algorithms to find centroids and derivations by using an algebra software.

\section{Definitions and Preliminary Results}
\label{sec:defprelimres}
\begin{defn}[\cite{GrazMakhlMeniniPanaite:bihom,Yongsheng}]
	A BiHom-Lie algebra over a field $ \K $ is a $ 4 $-tuple $ \left( L,[\cdot,\cdot],\alpha,\beta \right)  $, where $  L $ is
	a  $ \K $-linear space, $ \alpha\colon L\to L $ , $ \beta \colon L\to L$  and
	$ [\cdot,\cdot] \colon L\times L\to L$
	are linear maps, satisfying the following conditions, for all $ x,\,y,\, z\in L $:
	\begin{align}
	&\alpha\circ \beta=\beta\circ \alpha, \label{commute} \\
	%&\alpha\left( [x,y] \right) =[\alpha(x),\alpha(y)]  \quad  \text{  and   }   \quad   \beta\left( [x,y] \right) =[\beta(x),\beta(y)],\label{multiplicative}\text{(multiplicative)}\\
	&\left[\beta(x),\alpha(y)\right]=-\left[\beta(y),\alpha(x)\right] \ \ \text{\rm (skew-symmetry)} \label{skew}\\
	&\left[ \beta^{2}(x),[\beta(y),\alpha(z)]\right] +\left[\beta^{2}(y), [\beta(z),\alpha(x)]\right] + \left[ \beta^{2}(z),[\beta(x),\alpha(y)]\right] =0
\label{BiHomJacobi} \\
	&\qquad\qquad\qquad\qquad\qquad\qquad \ \ \text{\rm (BiHom-Jacobi identity).}   \nonumber
	\end{align}
A Bihom-Lie algebra is called a multiplicative Bihom-Lie algebra if	for any $ x,\,y\in L $,
\begin{equation}
 \alpha\left( [x,y] \right) =[\alpha(x),\alpha(y)]  \quad  \text{  and   }   \quad   \beta\left( [x,y] \right) =[\beta(x),\beta(y)].
 \label{multiplicative}
\end{equation}	
A  BiHom-Lie algebra is called a regular BiHom-Lie algebra if $ \alpha,\, \beta $ are bijective maps.
\end{defn}
In general for $ n $-dimensional case in terms of structure constants we have:
\begin{equation}\label{constant}
\begin{array}{lll}
[e_i,e_j] &=& \displaystyle{\sum_{s=1}^{n}} C_{ij}^{s} e_s,\, \\
\alpha(e_j)&=& \displaystyle{\sum_{s=1}^{n}} a_{sj}e_s \qquad \text{ and } \qquad
\beta(e_j)=\displaystyle{\sum_{s=1}^{n}} b_{sj}e_s.
\end{array}
\end{equation}
Substituting \eqref{constant} in the skew-symmetry identity \eqref{skew} yields	
	\begin{equation}\label{skewbih}	
\sum_{1\leq p,q\leq n }\left( b_{pi}a_{qj}+b_{pj}a_{qi}\right) C_{pq}^{s}	=0.
	\end{equation}	
Substituting \eqref{constant} in the BiHom-Jacobi identity \eqref{BiHomJacobi} yields
\begin{equation}\label{jacobbih}	
\sum_{1\leq p,q,s,l,s'\leq n }\left( b_{s'i}b_{qj}a_{sk}+
b_{s'j}b_{qk}a_{si}
+b_{s'k}b_{qi}a_{sj}\right) b_{ps'}C_{qs}^{l}C_{pl}^{r}	=0.
\end{equation}
Substituting  \eqref{constant} in the multiplicativity conditions \eqref{multiplicative} yields
\begin{equation}\label{mulbih}	
\begin{array}{lll}
\displaystyle{\sum_{1\leq k\leq n }} C_{ij}^{k}a_{sk} &=& \displaystyle{\sum_{1\leq p,q\leq n }} a_{pi}a_{qj}C_{pq}^{s} \\*[0,5cm]
\displaystyle{\sum_{1\leq k\leq n }} C_{ij}^{k}b_{sk} &=& \displaystyle{\sum_{1\leq p,q\leq n }} b_{pi}b_{qj}C_{pq}^{s}
\end{array}
\end{equation}	
for all $ i,j,k\in \{1,\dots,n\} $.
\begin{defn}		
A morphism $ f\colon \left(L,[\cdot,\cdot],\alpha,\beta \right)\to  \left(L',[\cdot,\cdot]',\alpha',\beta' \right)  $
of BiHom-Lie algebras is a linear map $ f\colon L\to  L'  $
such that $ \alpha'\circ f=f\circ \alpha  $,  $ \beta'\circ f=f\circ \beta $ and
\begin{equation}\label{isom}
f\left( [x,y] \right) =\left[ f(x),f(y)\right]' ,\quad \forall  x,\,y\in L .
\end{equation}
	In particular,  BiHom-Lie algebras $ \left(L,[\cdot,\cdot],\alpha,\beta \right) $ and $ \left(L',[\cdot,\cdot]',\alpha',\beta' \right)  $ are isomorphic if $ f $ is an isomorphism map.
\end{defn}	
\noindent Let $ \left(L,[\cdot,\cdot],\alpha,\beta \right) $ be $ n $-dimensional  BiHom-Lie algebra with ordered basis $ (e_1 ,\dots, e_n) $
 and  $ L' $ be  $ n $-dimensional vector spaces with ordered basis  $ (e'_1 ,\dots, e'_n )$. Let $ f\colon L\to L' $ be an isomorphism map.
%and $ (m_{ij}) $ be the matrix of $ f $ relative to the basis $ (e_1 ,\dots, e_n) $
%for $ L $.
Let $ \alpha'=f\alpha f^{-1} $ and $ \beta'=f\beta f^{-1} $. We set with respect to a basis  $ (e'_1 ,\dots, e'_n )$:
\begin{gather}
%\alpha'(e'_j)=\sum_{i=1}^{n}a'_{ij}e'_i,\quad' \beta'(e'_j)=\sum_{i=1}^{n}b'_{ij}e'_i,\\
f(e_j)=\sum_{i=1}^{n}f_{ij}e'_i,\\
[e'_i,e'_j]'=\sum_{k=1}^{n} C_{ij}^{'k}e'_{k},\, i,j\in\{1,\dots, n\}.
\end{gather}
Condition \eqref{isom}  translates to the following equation
\begin{equation}\label{isomB}
\sum_{k=1}^{n}C_{ij}^{k}f_{sk}=\sum_{1\leq p,q\leq n }f_{pi}f_{qj} C_{pq}^{'s},\quad i,j,s\in\{1,\dots, n\} .
\end{equation}
Then, if the previous condition satisfied, $ L' $ is a BiHom-Lie algebra isomorphic to $ L $. 	
	%A BiHom-Lie algebra is called an  abelian BiHom-Lie algebra if $ [x,y]=0 $	for any $ x,\,y\in L. $\\		

\begin{defn}[\cite{belhsine,Yongsheng}]
	Let $ \left(L,[\cdot,\cdot],\alpha,\beta \right) $ be a BiHom-Lie algebra. A subspace $ \mathfrak{h} $ of $ L $	is called a BiHom-Lie subalgebra   of  $ \left(L,[\cdot,\cdot],\alpha,\beta \right) $ if $ \alpha(\mathfrak{h})\subset \mathfrak{h} $,    $ \beta(\mathfrak{h})\subset \mathfrak{h} $  and $ [\mathfrak{h},\mathfrak{h}] \subset \mathfrak{h}.$  In particular, a BiHom-Lie subalgebra $\mathfrak{h}  $  is said to be an ideal of $\left(L,[\cdot,\cdot],\alpha,\beta \right)   $ if $ [\mathfrak{h},L] \subset \mathfrak{h}$ and $ [L,\mathfrak{h}]\subset\mathfrak{h} $.
%%%%%%%%%%%%%%%%%%%%%%%%%%%%

	If $ I $ is an ideal of  $ \left(L,[\cdot,\cdot],\alpha,\beta \right)  $,   then $ \left(L/ I,[\cdot,\cdot],\overline{\alpha},\overline{\beta}\right)  $, where $ [\overline{x},\overline{y}]=\overline{[x,y]} $, for all $ \overline{x},\, \overline{y}\in L/ I $  and $ \overline{\alpha},\,\overline{\beta}\colon  L/ I\to  L/ I $ naturally induced by $ \alpha $ and $ \beta $, inherits a BiHom-Lie algebra structure, which is named quotient BiHom-Lie algebra.
\end{defn}	
In the following, we give some examples and applications of ideals of BiHom-Lie algebras.
\begin{prop}	
If $ \left(L,[\cdot,\cdot],\alpha,\beta \right)  $   is a BiHom-Lie algebra, then   $ I=\ker \alpha + \ker \beta $ is an 	ideal of $ L $.	
%In this case the quotient BiHom-Lie algebra $L/ I  $ is regular.	
\end{prop}		
\begin{proof}
	By \eqref{commute} we get $ \alpha(I)\subset I $  and $ \beta(I)\subset I $.       By \eqref{multiplicative}   we obtain $ [I,L]\subset I $.
%\qed
\end{proof}
\begin{remq}
	 If $ L $   is finite-dimensional and $ \alpha $ (or $ \beta $) is diagonalizable, then there exist a subspace $ G $    such that $ L=I\oplus G $ and   $ \left(G,[\cdot,\cdot],\alpha_{/G},\beta_{/G} \right)  $   is a regular BiHom-Lie algebra.
\end{remq}

%%%%%%%%%%%%%%%%%%%%%%%%%%
\begin{defn}
Given a complex BiHom-Lie multiplicative algebra $L$, the center of $L$ is given by $ C(L)=\{x\in L\mid [x,y]=0\quad \forall y\in L\} $.
The descending central series of a BiHom-Lie algebra $L$ is given by the ideals
\[L^{0}=L;\quad L^k=[L,L^{k-1}],\quad k\geq 1.     \]
$ L $ is called nilpotent if $ L^{n}=\{0\} $ for some $ n\in \N $. If $ L^{n-1}\neq  \{0\}$, then $ L $ is said to be  $ n $-step nilpotent BiHom-Lie algebra.
The derived series of a BiHom-Lie algebra $ L $ is given by the ideals
$ L^{(0)}=L,\ L^{(k) }=[L^{(k-1) },L^{(k-1) }], k\geq 1.$
$ L $ is called solvable if $ L^{(n)}=\{0\} $ for some $ n\in \N $. If  $ L^{(n-1)}\neq  \{0\}$, then $ L $ is said to be
$ n $-step solvable BiHom-Lie algebra.	
\end{defn}
\begin{remq}
	 The center $ C(L) $ of $ L $ is not
 necessarily an ideal of $ L $. If $ \alpha $ and $ \beta $ are surjective then $ C(L) $ is an ideal of $L$.	 	
\end{remq}
\begin{defn}
	Let $ \left(L,[\cdot,\cdot],\alpha,\beta \right) $ be a BiHom-Lie algebra.   $ \left(L,[\cdot,\cdot],\alpha,\beta \right) $ is  called a simple  BiHom-Lie algebra if $ \left(L,[\cdot,\cdot],\alpha,\beta \right) $ has no proper ideals and is not abelian. $ \left(L,[\cdot,\cdot],\alpha,\beta \right) $ is called a semisimple  BiHom-Lie algebra if $ L $     is a direct sum of certain ideals.
\end{defn}
%\begin{prop}\label{regular}
%	Any finite-dimensional  simple BiHom-Lie algebra is regular.
%\end{prop}
%\begin{proof}
%	Let $ \left(L,[\cdot,\cdot],\alpha,\beta \right) $ be a finite-dimensional  simple BiHom-Lie algebra. With $\alpha\left( [x,y] \right) =[\alpha(x),\alpha(y)]  $ and     $ \beta\left( [x,y] \right) =[\beta(x),\beta(y)]  $
%	we can show  that $ \ker(\alpha) $ and  $ \ker(\beta) $  are ideals of $ \left(L,[\cdot,\cdot],\alpha,\beta \right) $.
%\end{proof}
\begin{prop}[\cite{GrazMakhlMeniniPanaite:bihom}]
	Let $ (L,[\cdot,\cdot]') $ be an ordinary Lie algebra over a field $ \K $ and let $ \alpha,\, \beta\colon L \to L $
	two commuting linear maps such that $ \alpha\left( [a,b]'\right) =\left[ \alpha(a) ,\alpha(b)\right]' $
	and $ \beta\left( [a,b]'\right) =\left[ \beta(a) ,\beta(b)\right]' $, for
	all $ a,\, b\in L $. Define the linear map $ [\cdot,\cdot]\colon L\times L \to L $, $ [a,b]=[\alpha(a),\beta(b)]' $,  for
	all $ a,\, b\in L $.
	Then $ L_{(\alpha,\beta)}:=\left( L,[\cdot,\cdot],\alpha,\beta \right)  $
	is a BiHom-Lie algebra, called the Yau twist of $ (L,[\cdot,\cdot]' ) $.	
\end{prop}	
\begin{exmpl}[\textbf{Heisengberg BiHom-Lie algebras}] \label{heisen}
Let $ (X,Y,Z) $	be a basis of a Heisenberg Lie algebra $ (\mathfrak{h}_{1},[\cdot,\cdot]') $ such that $$ [X,Y]'=Z,\, [X,Z]'=[Y,Z]'=0. $$
Let
$ \begin{pmatrix}
	b&0&0\\
	0&\frac{a}{b} &0\\
	0&0&a
\end{pmatrix} $  be the matrix of a linerar map $ \alpha\colon \mathfrak{h}_{1}\to \mathfrak{h}_{1} $ and  let
$ \begin{pmatrix}
	y&0&0\\
	0&\frac{x}{y} &0\\
	0&0&x
\end{pmatrix} $   be the matrix of a linerar map $ \beta\colon \mathfrak{h}_{1}\to \mathfrak{h}_{1} $ relative to a basis  $ (X,Y,Z) $ of $ \mathfrak{h}_{1} $. The  Yau twist $ H_{(\alpha,\beta)} $   of   $ \mathfrak{h}_{1} $  is called  Heisengberg BiHom-Lie algebra.   The bracket of Heisengberg BiHom-Lie algebra $ (\mathfrak{h}_{1},[\cdot,\cdot],\alpha,\beta) $ is gives by $ [X,Y]=b\frac{x}{y}Z $, $ [Y,X]=-y\frac{a}{b}Z $ and the other brother bracket are $ 0 $. For $ m> 1 $, we define
	Heisengberg BiHom-Lie algebra $ (\mathfrak{h}_{m},[\cdot,\cdot],\alpha,\beta) $ by
\begin{align*}
& [X_{i},Y_{i}]=b_{i}\frac{x}{y_{i}}Z,\quad [Y_{i},X_{i}]=-y_{i}\frac{a}{b_{i}}Z\quad  \forall i\in \{1,\dots,m\} \\
& \quad \quad \quad \quad \quad \quad \text{  and  the other brackets are  } 0;
\\
& \alpha=diag(b_{1},\dots,b_{m},\frac{a}{b_{1}},\dots,\frac{a}{b_{m}},a), \\
& \beta=diag(y_{1},\dots,y_{m},\frac{x}{y_{1}},\dots,\frac{x}{y_{m}},x) .
\end{align*}		
\end{exmpl}
\begin{prop}[\cite{Saadaoui:ClassmultiplsimpleBiHomLiealg}] \label{induced}
Let $ \left(G,[\cdot,\cdot],\alpha,\beta \right)  $ be a regular BiHom-Lie algebra.
Define the bilinear map $[\cdot,\cdot]'\colon G\times G\to G $ by
	\[ [x,y]'=[\alpha^{-1}(x),\beta^{-1}(y)],                           \]
	for all $ x,\, y\in G. $ Then $ \left(L,[\cdot,\cdot]'\right) $ is                                                                                                                                                                                                                                                                                                             a Lie algebra, which we call it the induced Lie algebra of  $ \left(L,[\cdot,\cdot],\alpha,\beta \right)  $.
\end{prop}
%\begin{defn}
%	The Lie algebra  $ \left(L,[\cdot,\cdot]'\right) $   is called the induced Lie algebra of  $ \left(L,[\cdot,\cdot],\alpha,\beta \right)  $.
%\end{defn}
\begin{prop}[\cite{Saadaoui:ClassmultiplsimpleBiHomLiealg}] \label{semisimple}
	The induced Lie algebra of the multiplicative simple  BiHom-Lie algebra is semisimple. There exist simple  ideal $ L_1 $  and an integer $ m\neq 2 $ such that \[ L=L_1\oplus\alpha(L_1)\oplus\cdots \oplus\alpha^{m-1}(L_1)=L_1\oplus\beta(L_1)\oplus\cdots \oplus\beta^{m-1}(L_1)    .\]	
	%	$ L_1,\,       \cdots,\, L_m  $  	
	% and  two cycles $ \sigma=(1\,\cdots\, m-1),\,\sigma'=(1\,i_1\,\cdots\, i_m)\in S_m $ such that $ L=L_1\oplus\cdots\oplus L_m  $, $ \alpha(L_i)=L_{\sigma(i)} $, $ \beta(L_i) =L_{\sigma'(i)}$ and $ \sigma\neq \sigma' $   (i.e $ L=L_1\oplus\alpha(L_1)\oplus\cdots \alpha^{m-1}(L_1)  $).
\end{prop}
\begin{prop}\label{regular}
	Any finite-dimensional multiplicative  simple BiHom-Lie algebra is regular.
\end{prop}
\begin{proof}
The statement holds since $ \ker \alpha $  and $ \ker \beta $  are ideals of the  simple BiHom-Lie algebra $ \left(L,[\cdot,\cdot],\alpha,\beta  \right)  $.
%\qed
\end{proof}
\begin{prop}\label{nilpotent}
 A regular multiplicative  BiHom-Lie algebra $ \left(G,[\cdot,\cdot],\alpha,\beta \right)  $ is nilpotent if and only if  the induced Lie  algebra  $ \left(G,[\cdot,\cdot]'\right) $ is nilpotent.
\end{prop}
\begin{prop}\label{solvable}
 A regular multiplicative  BiHom-Lie algebra $ \left(G,[\cdot,\cdot],\alpha,\beta \right)  $ is solvable if and only if  the induced Lie  algebra  $ \left(G,[\cdot,\cdot]'\right) $ is solvable. 	
\end{prop}
\begin{defn}
	A BiHom-Lie algebra $ L $ is said to  be  decomposable if it can be decomposed into the  direct  sum of two or more nonzero ideals. We say  $ L $ is indecompasable if it  is  not decomposable.
\end{defn}
\begin{prop}
Every  decomposable $ 2 $-dimensional multiplicative BiHom-Lie algebra $ L $ is nonregular and it satisfies $ L=[L,L]\oplus C(L) $.
\end{prop}

In this work, $ (L,[\cdot,\cdot],\alpha,\beta) $ denotes a multiplicative BiHom–Lie algebra (over field $ \C $), $ I=\ker \alpha + \ker \beta $,
$ G $ a BiHom-Lie subalgebra of $ L $ satisfying $ L=I\oplus G $ (if it exists),
 and $ \Omega =\{f\in End(L)\mid f\circ \alpha=\alpha\circ f,f\circ\beta=\beta\circ f\}   $.

%%%%%%%%%%%%%%%%%%%%%%%%%%%%%%%%%%%%%%%%
%%%%%%%%%%%%%%%%%%%%%%%
\section{Generalized derivations of BiHom-Lie algebras}
\label{sec:genderivBiHomLiealg}
%Let $ (L,[\cdot,\cdot],\alpha,\beta) $ be a BiHom–Lie algebra.
%We set
%\[\Omega =\{f\in End(L)\mid f\circ \alpha=\alpha\circ f,f\circ\beta=\beta\circ f\}    .\]
\begin{defn}[\cite{Yongsheng}]
	For any integer $ k,\, l $, a linear map $ D\colon L\to L $
	is called an
	$ \alpha^k\beta^l $-derivation of the BiHom-Lie  algebra $(L,[\cdot,\cdot ],\alpha,\beta)$,
	if
	$D\in \Omega$ and
	\begin{eqnarray*}
		D([x,y])&=&[D(x),\alpha^{k}\beta^{l}(y)]+[\alpha^{k}\beta^{l}(x),D(y)],
	\end{eqnarray*}
	for all  $ x,y \in L$. The set of all $ \alpha^k\beta^l $-derivations of a BiHom-Lie  algebra $(L,[\cdot,\cdot ],\alpha,\beta)$  is denoted by $Der_{\alpha^k\beta^l}(L)  ,$ and 	
 we denote by  $\displaystyle Der(L) $ the vector space spanned by the set $ \{d\in Der_{\alpha^k\beta^l}(L)\mid k,l\in \N\} $.	
	%We denote by $Der(L)  $ the vector space spanned by the set $ \{   d\mid  d\in ,v\in V     \} $
	% Denote by $\displaystyle Der(L)=\bigoplus_{0\leq k,l} Der_{\alpha^k\beta^l}(L)$ the set of derivations of the Bihom-Lie algebra $ (L,[\cdot,\cdot],\alpha,\beta) $.
\end{defn}

\begin{defn} Let $ (L,[\cdot,\cdot],\alpha,\beta) $ be a BiHom–Lie algebra and $\lambda,\mu,\gamma $
	 elements of $\C$.
	A linear map $d\in \Omega$ is a generalized 	$ \alpha^k\beta^l $-derivation or a	
	$ (\lambda,\mu,\gamma)$-$ \alpha^k\beta^l $-derivation of $L$ if for all $x,y\in L$ we have\[ \lambda\,	d([x,y])=\mu\, [d(x),\alpha^{k}\beta^{l}(y)]+ \gamma\,[\alpha^{k}\beta^{l}(x),d(y)]. \]
%	\textbf{ATTENTION THE SUM IS NOT  necessarily dierct}
	We denote the set of all $ (\lambda,\mu,\gamma)$-$ \alpha^k\beta^l $-derivations by $\displaystyle  Der_{\alpha^k\beta^l}^{(\lambda,\mu,\gamma)}(L)$
	 and $\displaystyle Der^{(\lambda,\mu,\gamma)}(L) $ the vector space spanned by  $ \{d\in Der_{\alpha^k\beta^l}^{(\lambda,\mu,\gamma)}(L)\mid k,l\in \N\} $.
	
	  %the set of generalized deivation of $ L $.
\end{defn}
\begin{lem}
For any $ D\in  Der_{\alpha^k\beta^l}^{(\lambda,\mu,\gamma)}(L)$ and $ D'\in  Der_{\alpha^s\beta^t}^{(\lambda',\mu',\gamma')}(L)$, their usual commutator defined by
	\begin{equation} \label{lieBracket}
	[D,D']'=D\circ D'-D'\circ D,
	\end{equation}	
	satisfies
	$[D,D']'\in  Der_{\alpha^{k+s}\beta^{l+t}}^{(\lambda\lambda',\mu\mu',\gamma\gamma')}(L). $
	%  $\left(  GDer(L),[\cdot,\cdot]' \right) $  is a Lie algebra.
\end{lem}	
\begin{proof}
The proof is similar to the one of $ Der(L) $ in \cite[Lemma 3.2]{Yongsheng}.	
%\qed
\end{proof}
%\begin{prop}
%	$ \left(Der(L),[\cdot,\cdot]' \right)  $   is a Lie algebra.
%\end{prop}
%Now, we consider the subspace  \[ IDer_{\alpha^k\beta^l}^{(\lambda,\mu,\gamma)}(L)=\{d\in  \Omega  \mid \lambda\,	d([x,u])=\mu\, [d(x),\alpha^{k}\beta^{l}(u)]+ \gamma\,[\alpha^{k}\beta^{l}(x),d(u)], \forall x\in L,u\in L^2   \}.\]
%%%%%%%%%%%%%%%%%%%%%%%%%%%%%
Let  us now classify the possible values of $ \lambda,\mu,\gamma \in\C $ for a linear map $ d\colon G\to G $ to be a $ (\lambda,\mu,\gamma)  $-$ \alpha^k\beta^l $-derivation of $ G $
\begin{lem}
	Let $ (G,[\cdot,\cdot],\alpha,\beta) $ be a  BiHom–Lie algebra
	such that the maps $\alpha$ and $\beta$ are surjective. Let  $\lambda,\mu,\gamma $
	be elements of $\C$.
	\begin{enumerate}[label=\upshape{\arabic*)},left=5pt]
		\item If $\lambda\neq 0$ and $\mu^2\neq \gamma^2$. Then $\displaystyle Der_{\alpha^k\beta^l}^{(\lambda,\mu,\gamma)}(G)=\displaystyle Der_{\alpha^k\beta^l}^{(\frac{ \lambda}{\mu+\gamma},1,0)}(G)$.
		\item If $\lambda\neq 0$, $\mu\neq 0$ and $\gamma=-\mu$. Then\\ $\displaystyle Der_{\alpha^k\beta^l}^{(\lambda,\mu,\gamma)}(G)=\displaystyle Der_{\alpha^k\beta^l}^{( 1,0,0)}(G)\cap Der_{\alpha^k\beta^l}^{( 0,1,-1)}(G)=Der_{\alpha^k\beta^l}^{( 1,1,-1)}(G)$.
		\item If $\lambda\neq 0$, $\mu=\gamma$ and $\mu\neq 0 $. Then $\displaystyle Der_{\alpha^k\beta^l}^{(\lambda,\mu,\gamma)}(G)=\displaystyle Der_{\alpha^k\beta^l}^{(\frac{ \lambda}{\mu},1,1)}(G)$.
		\item If $\lambda\neq 0$, $\mu=\gamma=0$. Then $\displaystyle Der_{\alpha^k\beta^l}^{(\lambda,\mu,\gamma)}(G)=\displaystyle Der_{\alpha^k\beta^l}^{(1,0,0)}(G)$.
		\item If $\lambda= 0$ and $\mu^2\neq \gamma^2$. Then $\displaystyle Der_{\alpha^k\beta^l}^{(\lambda,\mu,\gamma)}(G)=\displaystyle Der_{\alpha^k\beta^l}^{(0,1,0)}(G)$.
		\item If $\lambda= 0$, $ \mu \neq0 $ and $\mu= \gamma$. Then $\displaystyle Der_{\alpha^k\beta^l}^{(\lambda,\mu,\gamma)}(G)=\displaystyle Der_{\alpha^k\beta^l}^{(0,1,1)}(G)$.
		\item If $\lambda= 0$ and $\mu= -\gamma$. Then $\displaystyle Der_{\alpha^k\beta^l}^{(\lambda,\mu,\gamma)}(G)=\displaystyle Der_{\alpha^k\beta^l}^{(0,1,-1)}(G)$.
	\end{enumerate}	
\end{lem}
\begin{proof}
	Let $x,\,y\in G.$	Since $\alpha$ and $ \beta$ are surjective, there exists $a,\,b\in G $ such that $x=\beta(a) $, $y=\alpha(b)$.
 Suppose any $\lambda,\mu,\gamma\in\C  $ are given. Then for
 $ d\in Der_{\alpha^k\beta^l}^{(\lambda,\mu,\gamma)}(G) $
 and arbitrary $ a,b\in G $ we have	
 \begin{align*}
   \lambda\,	d([\beta(a),\alpha(b)])=\mu\, [d(\beta(a)),\alpha^{k}\beta^{l}(\alpha(b))]+ \gamma\,[\alpha^{k}\beta^{l}(\beta(a)),d(\alpha(b))] \\
   \lambda\,	d([\beta(b),\alpha(a)])=\mu\, [d(\beta(b)),\alpha^{k}\beta^{l}(\alpha(a))]+ \gamma\,[\alpha^{k}\beta^{l}(\beta(b)),d(\alpha(a))]
 \end{align*}
Thus, using $ d\circ \alpha=\alpha \circ d$, $ d\circ \beta=\beta \circ d$,  $ \alpha\circ\beta=\beta\circ\alpha $  and  \eqref{skew}, we have	
\begin{align*}
\lambda\,	d([\beta(a),\alpha(b)])=\mu\, [\beta(d(a)),\alpha^{k+1}\beta^{l}(b))]+ \gamma\,
[\alpha^{k}\beta^{l+1}(a)
,   \alpha(d(b)) ] \\
\lambda\,	d([\beta(b),\alpha(a)])=-\mu\, [\alpha^{k}\beta^{l+1}(a),\alpha(d(b))]
- \gamma\,[\beta(d(a))       ,\alpha^{k+1}\beta^{l}(b)]
\end{align*}
By summing the two previous equalities we obtain
\begin{equation*}
0=(\mu-\gamma)\left([\beta(d(a)),\alpha^{k+1}\beta^{l}(b)]-[\alpha^{k}\beta^{l+1}(a),\alpha(d(b))] \right).
\end{equation*}	
So, $(\mu-\gamma)\left([d(x),\alpha^{k}\beta^{l}(y)]-[\alpha^{k}\beta^{l}(x),d(y)] \right) =0  $.
Therefore, for $\mu\neq \gamma  $, $$ [d(x),\alpha^{k}\beta^{l}(y)]=[\alpha^{k}\beta^{l}(x),d(y)]. $$
Hence, applying $ d\in  Der_{\alpha^k\beta^l}^{(\lambda,\mu,\gamma)}(G)$ yields
$$\lambda d([x,y])= (\mu+\gamma) [d(x),\alpha^{k}\beta^{l}(y)] .$$
 The rest of the proof is easily deduced. %\qed
\end{proof}
%%%%%%%%%%%%%%%%%%%%%%
\begin{thm}\label{classGder}
	Let $ (G,[\cdot,\cdot],\alpha,\beta) $ be a  BiHom–Lie algebra
	such that the maps $\alpha$ and $\beta$ are surjective.
	%Assume that $[\beta(x),\alpha(y)] =0  $ for all $x,y\in L $ (i.e   $L$ is a BiHom-Lie algebra ).
	For any $\lambda,\mu,\gamma\in \C$
	there exists $\delta\in \C$ such that
	the subspace $ Der_{\alpha^k\beta^l}^{(\delta,\mu,\gamma)}(G)$
	is equal to one
	of the four following subspaces: 	
	\begin{multicols}{2}
		\begin{enumerate}[label=\upshape{\arabic*)},left=5pt]
			\item $\displaystyle Der_{\alpha^k\beta^l}^{(0,0,0)}(G)$;
			\item $\displaystyle Der_{\alpha^k\beta^l}^{(1,0,0)}(G)$;				
			\item $\displaystyle Der_{\alpha^k\beta^l}^{(\delta,1,0)}(G)$;
			\item $\displaystyle Der_{\alpha^k\beta^l}^{(\delta,1,1)}(G)$;
			\item $\displaystyle Der_{\alpha^k\beta^l}^{(1,1,-1)}(G)$;
			\item $\displaystyle Der_{\alpha^k\beta^l}^{(0,1,-1)}(G)$.			
		\end{enumerate}
	\end{multicols}
\end{thm}
%\begin{remq}
%	If $ \alpha $ or $ \beta $ is not surjective
% and $ I=\ker \alpha +\ker \beta $, $ L=I\oplus G $. Let $ D\in Der_{\alpha^k\beta^l}^{(\lambda,\mu,\gamma)}(L) $ Then, $D=D_{1}+D_{2} $  where $ D_{1} $ is a  $ (\lambda,0,0) $-$ \alpha^k\beta^l $-derivation of $ I $ satisfies $\lambda D_{1}([x,y])=\mu [D_{1}(x),\alpha^{k}\beta^{l}(y)] $ for all $ x\in I,\, y\in G $ and $D_{2}  $  is   $ (\lambda,\mu,\gamma) $-$ \alpha^k\beta^l $-derivation of the BiHom-Lie algebra $ G $.
%\end{remq}
\begin{exmpl}
Let $ H $ be a $3$-dimensional Heisenberg BiHom-Lie algebra (see Ex. \ref{heisen}).
\begin{align*}
Der_{\alpha^k\beta^l}^{(1,0,0)}(H)&= \left\lbrace \begin{pmatrix}
d_1&0&0\\
0&d_2&0\\
0&0&0
\end{pmatrix} \mid d_1,\, d_2\in \C      \right\rbrace; \\	
	%%%%%%%%%%%%%%%%%%%%%%%%%%%%%%%%%
Der_{\alpha^k\beta^l}^{(\delta,1,0)}(H)&= \left\lbrace \begin{pmatrix}
d_{1}&0&0\\
0&d_{1}\frac{a^{r}x^{l}}{ b^{2r}y^{2l}} &0\\
0&0&d_{1}\frac{ a^{r}x^{l}}{\delta b^{r}y^{l} }
\end{pmatrix} \mid  d_1\in \C      \right\rbrace ;\\
%%%%%%%%%%%%%%%%%%%%%%%%%%%%%%%%%%%%%%%%
%%%%%%%%%%%%%%%%%%%%%%%%%%%	
Der_{\alpha^k\beta^l}^{(\delta,1,1)}(H)&= \left\lbrace \begin{pmatrix}
d_{1}&0&0\\
0&d_2&0\\
0&0&\frac{d_2 b^{2r}y^{2l}+d_{1}a^{r}x^{l}}{\delta b^{r}y^{l}}
\end{pmatrix} \mid d_{1}, d_2\in \C      \right\rbrace; \\
%%%%%%%%%%%%%%%%%%%%%%%%%%%%%%%%%%
	%%%%%%%%%%%%%%%%%%%%%%%%%%%%%
Der_{\alpha^k\beta^l}^{(0,1,1)}(H)&= \left\lbrace \begin{pmatrix}
d_{1}&0&0\\
0&-d_1\frac{ a^{r}x^{l}}{ b^{2r}y^{2l}} &0\\
0&0&d_{3}
\end{pmatrix} \mid d_{1}, d_3\in \C      \right\rbrace; \\	
%%%%%%%%%%%%%%%%%%%%%%%%%%%%%
Der_{\alpha^k\beta^l}^{(1,1,-1)}(H)&= \left\lbrace \begin{pmatrix}
d_{1}&0&0\\
0&d_{1}\frac{a^{r}x^{l}}{ b^{2r}y^{2l}} &0 \\
0&0&0
\end{pmatrix} \mid  d_1\in \C      \right\rbrace;\\
	%%%%%%%%%%%%%%%%%%%%%%%%%%%%%
Der_{\alpha^k\beta^l}^{(0,1,-1)}(H)&= \left\lbrace \begin{pmatrix}
d_{1}&0&0\\
0&d_1\frac{ a^{r}x^{l}}{ b^{2r}y^{2l}} &0\\
0&0&d_{3}
\end{pmatrix} \mid d_{1}, d_3\in \C      \right\rbrace.
%%%%%%%%%%%%%%%%%%%%%%%%%%%	
	\end{align*}	
\end{exmpl}
Next proposition allows to extend some results from
 \cite{DoradodoGD} to BiHom-Lie case.
\begin{prop}\label{LieBi}
Any  $ (\lambda,\mu,\gamma)$-$\alpha^0\beta^0$-derivation  of regular  multiplicative
BiHom-Lie algebra $\left(G,[\cdot,\cdot],\alpha,\beta \right)$ is a $(\lambda,\mu,\gamma)$-derivation  of induced Lie algebra
$ \left(G,[\cdot,\cdot]' \right)  $.
\end{prop}
Now we will discuss in detail the possible Theorem \ref{classGder} for a finite-dimensional BiHom-Lie algebra $ L $, and we give the connection between the generalized derivation of the type studied in \cite{belhsine} and the generalized derivations of the type studied in this work.
\begin{enumerate}[label=\upshape{\arabic*)},left=5pt]
\item  \underline{$\displaystyle Der_{\alpha^k\beta^l}^{(0,0,0)}(L)=\Omega$}.
We have $ \Omega $  is a Lie algebra, where Lie bracket is given by \eqref{lieBracket}.
\item \underline{ $\displaystyle Der_{\alpha^k\beta^l}^{(1,0,0)}(L)=\left\lbrace d\in \Omega\mid d(L^{2})=0 \right\rbrace $}
and therefore its dimension is $$ dim Der_{\alpha^k\beta^l}^{(1,0,0)}(L) = codim L^{2}\, dim L. $$
If the BiHom-Lie algebra $ L $ is simple, then $ Der_{\alpha^0\beta^0}^{(1,0,0)}(L) =\{0\}  $.
\item  \underline{$\displaystyle Der_{\alpha^k\beta^l}^{(\delta,1,0)}(L)$:}
\begin{enumerate}[label=\upshape{(\alph*)},left=0pt]
\item  If $ \delta=0,$ then $\displaystyle Der_{\alpha^k\beta^l}^{(0,1,0)}(L)=\left\lbrace d\in \Omega\mid d(L)\subset C(\alpha^{k}\beta^{l}(L))\right\rbrace $, where
    $$ C(\alpha^{k}\beta^{l}(L))=\left\lbrace x\in L\mid [x,y]=0; \forall y\in  \alpha^{k}\beta^{l}(L)\right\rbrace ,$$
    is the centralizer of $\alpha^{k}\beta^{l}(L)$.
    Therefore, $$ \dim Der_{\alpha^k\beta^l}^{(0,1,0)}(L)=\dim L\, \dim C(\alpha^{k}\beta^{l}(L)).$$

If the BiHom-Lie algebra $ L $ is simple, then $\displaystyle Der_{\alpha^0\beta^0}^{(0,1,0)}(L)=\{0\}$.
\item  If $ \delta=1$, then $\displaystyle Der_{\alpha^k\beta^l}^{(1,1,0)}(L)$ is the  $\alpha^k\beta^l$-centroid of $L$  denoted $\varGamma_{\alpha^k\beta^l}(L)$.  We denote by  $\displaystyle \varGamma(L) $ the vector space spanned by the set $ \{d\in \varGamma_{\alpha^k\beta^l}(L)\mid k,l\in \N\} $.
\begin{enumerate}[label=\upshape{(\roman*)},left=0pt]
	\item  If $ \phi\in   \varGamma_{\alpha^k\beta^l}(G)$  and $ d\in Der_{\alpha^s\beta^r}(G)  $, then $ \phi\circ d $ is a $ \alpha^{k+s}\beta^{r+l} $-derivation of $ G $.
\item  $\varGamma_{\alpha^k\beta^l}(L) \cap Der_{\alpha^k\beta^l}(L)= CDer_{\alpha^k\beta^l}(G)$, where $ CDer_{\alpha^k\beta^l}(G) $ is the set of $ \alpha^k\beta^l $-central derivations defined by \[CDer_{\alpha^k\beta^l}(G)= Der_{\alpha^k\beta^l}^{(1,0,0)}(G)\bigcap  Der_{\alpha^k\beta^l}^{(0,1,0)}(G).\]
\item  For any $d\in Der_{\alpha^k\beta^l}(G)  $ and $ \phi\in   \varGamma_{\alpha^t\beta^s}(G) $ one has	
\begin{list}{$\bullet$}{}
  \item The composition $d\circ \phi$ is in $\varGamma_{\alpha^{k+t}\beta^{l+s}}(G)  $ if and only if $ \phi\circ d $ is a central
  derivation of $ L $;	
    \item  The composition $d\circ \phi$ is a $\alpha^{k+t}\beta^{l+s}  $-derivation of $ G $ if and only if $ [d, \phi ]'$ is a
    $\alpha^{k+t}\beta^{l+s}  $-central derivation of $ G $. (see \cite{AlmutariAhmad:CentrquasicentrfdimLeibnizalg} for  the Leibniz case and \cite{AbdulkarimAssociative}  for the  associative algebras case).
\end{list}
\end{enumerate}
%\end{prop}
%Let $\left(G,[\cdot,\cdot],\alpha,\beta \right)   $  be a regular multiplicative BiHom-Lie algebra.	
Suppose that $ L $  admits a generalized derivation $D\in  Der_{\alpha^0\beta^0}^{(1,1,0)}(L) $. If $ \lambda\in \sigma(D) $ is an eigenvalue of $ D $, then the corresponding generalized eigenspace $ L_{\lambda} $ is an ideal of $ L $.
Moreover, the generalized eigenspace decomposition $ L=\oplus_{\lambda\in\sigma(D)} L_{\lambda}$ is given in terms of ideals of $ L $.
Suppose that the BiHom-Lie algebra $ L $ is simple, then $Der_{\alpha^0\beta^0}^{(1,1,0)}(L)$ is the one–dimensional BiHom-Lie algebra containing multiples of the identity operator.		
\item For $ \delta\notin \{0,1\} $.
%\begin{prop}
Suppose that $ G $ is non-abelian. Then, by Proposition \ref{nilpotent}, Proposition \ref{LieBi} and \cite[Proposition 2.19]{DoradodoGD}, the following statements are equivalent:		
	\begin{enumerate}[label=\upshape{(\roman*)},left=-7pt]
		\item $ G $ admits an invertible generalized derivation $ D\in Der_{\alpha^0\beta^0}^{(\delta,1,0)}(G) $.
		\item $ G $ is at most a $ 2 $-step nilpotent BiHom-Lie algebra.
		\item $ G $ admits an invertible semisimple
		generalized derivation \\ $ D\in Der_{\alpha^0\beta^0}^{(\delta,1,0)}(G) $ with minimal polynomial  $ q(x)=(x-\delta^{-1})(x-1).$
	\end{enumerate}
%\end{prop}
%%%%%%%%%%%%%%%%%%%%%%%%%
%%%%%%%%%%%%%%%%%
%By Proposition \ref{semisimple}, Proposition\ref{LieBi} and
%\cite[Example $ 1 $. ]{HrivnakGD}, we have
%\begin{prop}
%	If the  regular  BiHom-Lie algebra $\left(L,[\cdot,\cdot],\alpha,\beta \right)   $
%	is simple then  $Der_{\alpha^0\beta^0}(L)\cong L$,
%	\[ Der_{\alpha^0\beta^0}^{(1,0,0)}(L)=Der_{\alpha^0\beta^0}^{(0,1,0)}(L)=\{0\}   \]
%	and 	
%	$Der_{\alpha^0\beta^0}^{(1,1,0)}(L)$      is the one–dimensional Lie algebra containing multiples of the identity
%	operator.
%\end{prop}

\end{enumerate}
%%%%%%%%%%%%%%%%%%%%%%%%	
\item\underline{  $\displaystyle Der_{\alpha^k\beta^l}^{(\delta,1,1)}(L)$:}	
\begin{enumerate}[label=\upshape{(\alph*)},left=0pt]
\item For $ \delta=0 $ 	we have a Lie algebra
\begin{multline*}
Der^{(0,1,1)}(L) =
\{d\in \Omega \mid \exists k,l\in \N: \\
\quad [d(x),\alpha^{k}\beta^{l}(y)]=-[\alpha^{k}\beta^{l}(x),d(y)],\forall x,y\in L\}.
\end{multline*}
If the BiHom-Lie algebra $ L $ is simple, then
by Proposition \ref{semisimple}, Proposition \ref{LieBi} and
\cite[Corollary 2.10, Corollary 2.11]{DoradodoGD}, we have
 $Der_{\alpha^0\beta^0}^{(0,1,1)}(L)=\{0\}$.
If the simple BiHom-Lie algebra $ L $  admits an invertible generalized derivation $D\in Der_{\alpha^0\beta^0}^{(0,1,1)}(L)$, then $L$ is solvable.
%%%%%%%%%%%
\item For $ \delta=1 $.  We get the Lie algebra of derivations of L: \\ $Der_{\alpha^k\beta^k}^{(1,1,1)}(L)=Der_{\alpha^k\beta^k}(L)  $ and $ \left( Der(L),[\cdot,\cdot]' \right) $ is a Lie algebra ($\displaystyle Der(L) $ the vector space spanned by  $ \{d\in Der_{\alpha^k\beta^l}(L)\mid k,l\in \N\} $).
%%%%%%%%%%%%%%%%%%%%%%%%%%%%%%%%%%%%%%%%%%%%%
\item For $ \delta \notin\{-1,0,1,2\} $.  When $ G $ admits an invertible semisimple generalized $ D\in Der_{\alpha^0\beta^0}^{(\delta,1,1)}(G) $
by Proposition \ref{solvable}, Proposition \ref{LieBi} and
\cite[Proposition 2.8]{DoradodoGD},  $ G $ is at most a $ 3 $-step solvable BiHom-Lie algebra. When the  invertible semisimple generalized $ D $ has only two different eigenvalues,
by Proposition \ref{nilpotent}, Proposition \ref{LieBi} and
\cite[Lemma 2.2]{DoradodoGD},
 $ G $ is at most a $ 2 $-step nilpotent BiHom-Lie algebra.
\end{enumerate}
%%%%%%%%%%%%%%%%%%%%%%%%	
\item\underline{$\displaystyle Der_{\alpha^k\beta^l}^{(1,1,-1)}(G)$:} We have 	
\begin{multline*}
Der_{\alpha^k\beta^l}^{(1,1,-1)}(G)=Der_{\alpha^k\beta^l}^{(0,1,-1)}(G)\cap Der_{\alpha^k\beta^l}^{(1,0,0)}(G)\\
=\left\lbrace d \in \Omega\mid d([x,y])= 0= [d(x),\alpha^{k}\beta^{l}(y)]=[\alpha^{k}\beta^{l}(x),d(y)] \right\rbrace.
\end{multline*}
Then $\displaystyle Der_{\alpha^k\beta^l}^{(1,1,-1)}(G)$ is the set of $\alpha^k\beta^l$-central derivations of $ G $. Define the bilinear map
\begin{equation}\label{Qcent}
\mu\colon \Omega\times \Omega\to \Omega,\, \mu(f,g)=\frac{1}{2}\left(f\circ g+g \circ f \right).
\end{equation}
Then $\left( CDer(G),\mu\right)   $ is a Jordan algebra.
%%%%%%%%%%%%%%%%%%%%%%%%	
\item\underline{$\displaystyle Der_{\alpha^k\beta^l}^{(0,1,-1)}(L)$:} We have
\[   \displaystyle Der_{\alpha^k\beta^l}^{(0,1,-1)}(G)    = \left\lbrace  d\in   \Omega\mid        [d(x),\alpha^{k}\beta^{l}(y)]=[\alpha^{k}\beta^{l}(x),d(y)]          \right\rbrace             .\]
Then    $\displaystyle Der_{\alpha^k\beta^l}^{(0,1,-1)}(L)$ is called $\alpha^k\beta^l$-quasi-centroid of $L$ and denoted $QC_{\alpha^k\beta^l}(L)$. With the bilinear map $ \mu $ defined in \eqref{Qcent}, we have that $ \left(QC_{\alpha^k\beta^l}(G), \mu \right)  $ is a Jordan algebra.
\end{enumerate}

We end this section with a construction of a BiHom-Lie algebra from an extension of a Lie algebra $ L $ by a $ (b,a,a) $-derivation of $ L  $.
\begin{prop}
 Let $ (L,[\cdot,\cdot]') $ be a Lie algebra and
 $ D\in End(L) $  be a non-zero $ (b, a, a) $-derivation and let $ \alpha\colon L\oplus\C D\to  L\oplus\C D$ and
 $ \beta\colon L\oplus\C D\to  L\oplus\C D$ defined respectivly by $ \alpha(x+\lambda D)=x+\lambda a D $ and $ \beta(x+\lambda D)=x+\lambda b  $; $ x\in L,\, \lambda\in \C $. Let
 Define the bilinear map $ [\cdot,\cdot]\colon L\oplus\C D\times L\oplus\C D\to L\oplus\C D $, $ [x+\lambda d,y+\mu d]=[x,y]'-\mu b d(x)+\lambda a d(y)  $. Then
 $ \left(L\oplus\C D,[\cdot,\cdot],\alpha,\beta \right)  $ is a  BiHom-Lie algebra.
\end{prop}

%%%%%%%%%%%%%%%%%%%%%%%%%%%%%

%%%%%%%%%%%%%%%%%%%%%%%%%%%%%
%By  \textbf{Proposition}\ref{LieBi} and
%\cite[\textbf{Proposition 3.2} ]{DoradodoGD}, we have
%\begin{prop}
%Let $ L $ be a non-abelian regular multiplicative BiHom-Lie algebra. If $ L $ admits a periodic generalized derivation
%$D\in Der_{\alpha^0\beta^0}^{(-1,1,1)}(L)   $
%of order $ m $, then $  m $ is a multiple of three.	
%\end{prop}
%%%%%%%%%%%%%%%%%%%%%%%%%%%%%
%By  \textbf{Proposition}\ref{LieBi} and
%\cite[\textbf{Theorem 3.5} ]{DoradodoGD}, we have
%\begin{prop}
%	Let $ L $ be a  regular multiplicative BiHom-Lie algebra.
%Then the following statements are equivalent:
%\begin{enumerate}[label=\upshape{(\arabic*)},left=0pt]
%	\item $ L $ admits an invertible periodic generalized derivation
%	$D\in Der_{\alpha^0\beta^0}^{(-1,1,1)}(L)   $.
%	\item $ L $ is  anti-triangularly graded.	
%\end{enumerate}	
%\end{prop}
%%%%%%%%%%%%%%%%%%%%%%%%%%%%%%%%%%%%%%

\section{ Classification of multiplicative $ 2 $-dimensional \\ BiHom-Lie algebras}
\label{sec:classmultiplic2dimBiHomLiealg}
In this section, we aim to classify $ 2 $-dimensional non-trivial
BiHom-Lie algebras.
 An $ n $-dimensional multiplicative  BiHom-Lie algebra is identified to its structure constants with respect to a fixed basis. It turns out that the
axioms of  multiplicative  BiHom-Lie algebra structure translate to a system of polynomial equations that define
the algebraic variety of $ n $-dimensional  multiplicative BiHom-Lie algebra which is embedded into $ \K^{n^{3}+2n^{2}} $.
The classification requires to solve this algebraic system. The calculations are handled using
a computer algebra system. For $ n=2 $, we include in the following an outline of the computation.
\begin{enumerate}[label=\upshape{\arabic*.},left=0pt]
\item \label{solut1} Solving \eqref{commute}, we obtain the following solutions:	
\begin{enumerate}[label=\upshape{\ref{solut1}\arabic*.},left=0pt]
	\item \label{enumitem1:solsalphbet} \qquad $ \alpha=\begin{pmatrix}
	a&0\\
	0&b
	\end{pmatrix},
\qquad \beta=\begin{pmatrix}
	x&0\\
	0&y
	\end{pmatrix} $;
	\item \label{enumitem2:solsalphbet} \qquad $ \alpha=\begin{pmatrix}
	a&0\\
	0&a
	\end{pmatrix},
\qquad \beta=\begin{pmatrix}
	x&1\\
	0&x
	\end{pmatrix} $;
	\item \label{enumitem3:solsalphbet} \qquad $ \alpha=\begin{pmatrix}
	a&1\\
	0&a
	\end{pmatrix},
\qquad \beta=\begin{pmatrix}
	x&z\\
	0&x
	\end{pmatrix} $.	
\end{enumerate}	
	\item \label{solut2} For each solution in \ref{solut1}, we provide a list of non-trivial $ 2 $-dimensional multiplicative  BiHom-Lie algebras.
We solve the system of equations \eqref{skewbih}, \eqref{jacobbih} and \eqref{mulbih} such that
\begin{list}{$\bullet$}{}	
	\item $  a_{12}=a_{21}=0 $,  $b_{12}=b_{21}=0  $, for \ref{enumitem1:solsalphbet}
	\item  $  a_{12}=a_{21}=0 $, $ a_{22}=a_{11} $,  $b_{21}=0,\,b_{12}=1,\,b_{22}=b_{11}  $, for \ref{enumitem2:solsalphbet}
	\item $ a_{21}=0,\, a_{12}=1,\, $ $ a_{22}=a_{11} $,  $b_{21}=0,\,b_{22}=b_{11}  $, for \ref{enumitem3:solsalphbet}
\end{list}
\item  Fix a BiHom-Lie algebra $ L $ in \ref{solut2} and 	solving the equation \eqref{isomB}  such that $ C_{ij}^{k} $ are the structure constants corresponding to $ L $  and  $ f_{12} = f_{21} = 0   $ (resp. $f_{11} = f_{22} = 0   $) if $  [L,L]\neq<e_2> $ (resp. $  [L,L]=<e_{2}> $ ).
\end{enumerate}
Therefore, we get the following result.
\begin{prop}\label{A1}
 Every  $ 2 $-dimensional multiplicative  BiHom-Lie algebra is isomorphic to one of the following non-isomorphic BiHom-Lie algebras: each  algebra is denoted by $ L_{j}^{i} $ where $ i $ is related to the couple $ (\alpha,\beta) $, $ j $  is the number.
		%============Alg 15 cas 1==============	
	%============Alg 15 cas 1==============	
%%%%%%%%%%%%%%%%%%%%%

\begin{align*}
L_{1}^{1}:[e_{1},e_{1}]&=e_1,&     [e_{1},e_{2}]&=e_{1},
&
[e_2,e_{1}]&=z_1 e_{1}
,&       [e_2,e_{2}]&=0,&\\
\alpha(e_{1})&=0,& \alpha(e_{2})&=be_{2},&\beta(e_{1})&=0,& \beta(e_{2})&=ye_{2}.&   	
\end{align*}
%%%%%%%%%%%%%%%%%%%%%				
\begin{align*}
L_{2}^{1}:[e_{1},e_{1}]&=e_1,&     [e_{1},e_{2}]&=0,
&
[e_2,e_{1}]&= e_{1}
,&       [e_2,e_{2}]&=0,&\\
\alpha(e_{1})&=0,& \alpha(e_{2})&=be_{2},&\beta(e_{1})&=0,& \beta(e_{2})&=ye_{2}.&   	
\end{align*}
%%%%%%%%%%%%%%%%%%%%%				
\begin{align*}
L_{3}^{1}:[e_{1},e_{1}]&=e_1,&     [e_{1},e_{2}]&=0,
&
[e_2,e_{1}]&= 0
,&       [e_2,e_{2}]&=0,&\\
\alpha(e_{1})&=0,& \alpha(e_{2})&=be_{2},&\beta(e_{1})&=0,& \beta(e_{2})&=ye_{2}.&   	
\end{align*}
\begin{align*}
L_{4}^{1}:[e_{1},e_{1}]&=0,&     [e_{1},e_{2}]&=e_{1},
&
[e_2,e_{1}]&= z_1e_{1}
,&       [e_2,e_{2}]&=0,&\\
\alpha(e_{1})&=0,& \alpha(e_{2})&=be_{2},&\beta(e_{1})&=0,& \beta(e_{2})&=ye_{2}.&   	
\end{align*}
%============Alg 15 cas 1==============	
%%%%%%%%%%%%%%%%%%%%%				
\begin{align*}
L_{5}^{1}:[e_{1},e_{1}]&=0,&     [e_{1},e_{2}]&=0,
&
[e_2,e_{1}]&= e_{1}
,&       [e_2,e_{2}]&=0,&\\
\alpha(e_{1})&=0,& \alpha(e_{2})&=be_{2},&\beta(e_{1})&=0,& \beta(e_{2})&=ye_{2}.&   	
\end{align*}

%============Alg 65 cas 1==============	
%%%%%%%%%%%%%%%%%%%%%
\begin{align*}
L_{1}^{2}:[e_{1},e_{1}]&= e_{1},&     [e_{1},e_{2}]&=0,
&
[e_2,e_{1}]&=0
,&       [e_2,e_{2}]&=0,&\\
\alpha(e_{1})&= e_{1},& \alpha(e_{2})&=be_{2},&\beta(e_{1})&=0,& \beta(e_{2})&=ye_{2}.&   	
\end{align*}
%============Alg 66 cas 1==============y_1\neq 0	
%%%%%%%%%%%%%%%%%%%%%
%\begin{align*}
%L_{1}^{3}:[e_{1},e_{1}]&=0 ,&     [e_{1},e_{2}]&=e_{1},
%&
%[e_2,e_{1}]&=z_{1}e_{1}
%,&       [e_2,e_{2}]&=0,&\\
%\alpha(e_{1})&= 0,& \alpha(e_{2})&=e_{2},&\beta(e_{1})&=0,& \beta(e_{2})&=ye_{2}.&   	
%\end{align*}
%============Alg 66 cas 1==============y_1= 0	
%%%%%%%%%%%%%%%%%%%%%
%\begin{align*}
%L_{2}^{3}:[e_{1},e_{1}]&=0 ,&     [e_{1},e_{2}]&=0,
%&
%[e_2,e_{1}]&=e_{1}
%,&       [e_2,e_{2}]&=0,&\\
%\alpha(e_{1})&= 0,& \alpha(e_{2})&=e_{2},&\beta(e_{1})&=0,& \beta(e_{2})&=ye_{2}.&   	
%\end{align*} 		
%============Alg 68 cas 1==============	
%%%%%%%%%%%%%%%%%%%%%
\begin{align*}
L_{1}^{3}:[e_{1},e_{1}]&=0 ,&     [e_{1},e_{2}]&=e_{1},
&
[e_2,e_{1}]&=0
,&       [e_2,e_{2}]&=0,&\\
\alpha(e_{1})&=a e_{1},& \alpha(e_{2})&=e_{2},&\beta(e_{1})&=0,& \beta(e_{2})&=ye_{2}.&   	
\end{align*} 		
%============Alg 67 cas 1==============x1\neq	 0;y1\neq	 0;
%%%%%%%%%%%%%%%%%%%%%

\begin{align*}
L_{1}^{4}:[e_{1},e_{1}]&=e_1,&     [e_{1},e_{2}]&=e_{1},
&
[e_2,e_{1}]&=0
,&       [e_2,e_{2}]&=0,&\\
\alpha(e_{1})&=e_{1},& \alpha(e_{2})&=e_{2},&\beta(e_{1})&=0,& \beta(e_{2})&=ye_{2}.&   	
\end{align*}
%%%%%%%%%%%%%%%%%%%%%
%============Alg 67 cas 1==============x1\neq	 0;y1=	 0;
%%%%%%%%%%%%%%%%%%%%%
%\begin{align*}
%L_{2}^{5}:[e_{1},e_{1}]&=e_1,&     [e_{1},e_{2}]&=0,
%&
%[e_2,e_{1}]&=0
%,&       [e_2,e_{2}]&=0,&\\
%\alpha(e_{1})&=e_{1},& \alpha(e_{2})&=e_{2},&\beta(e_{1})&=0,& \beta(e_{2})&=ye_{2}.&   	
%\end{align*}
%%%%%%%%%%%%%%%%%%%%%
%============Alg 67 cas 1==============x1=	 0
%%%%%%%%%%%%%%%%%%%%%
%\begin{align*}
%L_{3}^{5}:[e_{1},e_{1}]&=0,&     [e_{1},e_{2}]&=e_1,
%&
%[e_2,e_{1}]&=0
%,&       [e_2,e_{2}]&=0,&\\
%\alpha(e_{1})&=e_{1},& \alpha(e_{2})&=e_{2},&\beta(e_{1})&=0,& \beta(e_{2})&=ye_{2}.&   	
%\end{align*}
%%%%%%%%%%%%%%%%%%%%% 		
%============Alg 73 cas 1==============y1\neq	 0
%%%%%%%%%%%%%%%%%%%%%
%\begin{align*}
%L_{1}^{6}:[e_{1},e_{1}]&=0,&     [e_{1},e_{2}]&=e_{1},
%&
%[e_2,e_{1}]&=z_{1}e_{1}
%,&       [e_2,e_{2}]&=0,&\\
%\alpha(e_{1})&=0,& \alpha(e_{2})&=be_{2},&\beta(e_{1})&=0,& \beta(e_{2})&=e_{2}.&   	
%\end{align*}
%============Alg 73 cas 1==============y1=	 0
%%%%%%%%%%%%%%%%%%%%%
%\begin{align*}
%L_{2}^{6}:[e_{1},e_{1}]&=0,&     [e_{1},e_{2}]&=0,
%&
%[e_2,e_{1}]&=e_{1}
%,&       [e_2,e_{2}]&=0,&\\
%\alpha(e_{1})&=0,& \alpha(e_{2})&=be_{2},&\beta(e_{1})&=0,& \beta(e_{2})&=e_{2}.&   	
%\end{align*}		
%%%%%%%%%%%%%  Alg 76  %%%%%%%%%%%%%%
\begin{align*}
L_{1}^{5}:[e_{1},e_{1}]&=e_{1},&     [e_{1},e_{2}]&=0,
&
[e_2,e_{1}]&=0
,&       [e_2,e_{2}]&=0,&\\
\alpha(e_{1})&=0,& \alpha(e_{2})&=be_{2},&\beta(e_{1})&=e_{1},& \beta(e_{2})&=ye_{2}.&   	
\end{align*}		
 %%%%%%%%%%%%%  Alg 80  %%%%%%%%%%%%%%
\begin{align*}
L_{1}^{6}:[e_{1},e_{1}]&=e_{1},&     [e_{1},e_{2}]&=0,
&
[e_2,e_{1}]&=e_{1}
,&       [e_2,e_{2}]&=0,&\\
\alpha(e_{1})&=0,& \alpha(e_{2})&=be_{2},&\beta(e_{1})&=e_{1},& \beta(e_{2})&=e_{2}.&   	
\end{align*}
%\begin{align*}
%L_{2}^{8}:[e_{1},e_{1}]&=e_{1},&     [e_{1},e_{2}]&=0,
%&
%[e_2,e_{1}]&=0
%,&       [e_2,e_{2}]&=0,&\\
%\alpha(e_{1})&=0,& \alpha(e_{2})&=be_{2},&\beta(e_{1})&=e_{1},& \beta(e_{2})&=e_{2}.&   	
%\end{align*}
%\begin{align*}
%L_{3}^{8}:[e_{1},e_{1}]&=0,&     [e_{1},e_{2}]&=0,
%&
%[e_2,e_{1}]&=e_{1}
%,&       [e_2,e_{2}]&=0,&\\
%\alpha(e_{1})&=0,& \alpha(e_{2})&=be_{2},&\beta(e_{1})&=e_{1},& \beta(e_{2})&=e_{2}.&   	
%\end{align*}		
 %%%%%%%%%%%%%  Alg 93  %%%%%%%%%%%%%%
\begin{align*}
L_{1}^{7}:[e_{1},e_{1}]&=0,&     [e_{1},e_{2}]&=0,
&
[e_2,e_{1}]&=e_1
,&       [e_2,e_{2}]&=0,&\\
\alpha(e_{1})&=0,& \alpha(e_{2})&=be_2,&\beta(e_{1})&=xe_{1},& \beta(e_{2})&=e_2.&   	
\end{align*}		
%%%%%%%%%%%%%  Alg 96  %%%%%%%%%%%%%%
\begin{align*}
L_{1}^{8}:[e_{1},e_{1}]&=0,&     [e_{1},e_{2}]&=e_1,
&
[e_2,e_{1}]&=-\frac{x}{a}e_1
,&       [e_2,e_{2}]&=0,&\\
\alpha(e_{1})&=ae_1,& \alpha(e_{2})&=e_2,&\beta(e_{1})&=xe_{1},& \beta(e_{2})&=e_2.&   	
\end{align*}
%%%%%%%%%%%%%  Alg 101  %%%%%%%%%%%%%%
%\begin{align*}
%L_{1}^{9}:[e_{1},e_{1}]&=e_{1},&     [e_{1},e_{2}]&=0,
%&
%[e_2,e_{1}]&=0
%,&       [e_2,e_{2}]&=0,&\\
%\alpha(e_{1})&=e_1,& \alpha(e_{2})&=0,&\beta(e_{1})&=0,& \beta(e_{2})&=ye_2.&   	
%\end{align*}		
%%%%%%%%%%%%%  Alg 108 %%%%%%%%%%%%%% x1\neq0;t2\neq0
\begin{align*}
L_{1}^{9}:[e_{1},e_{1}]&=e_1,&     [e_{1},e_{2}]&=0,
&
[e_2,e_{1}]&=0
,&       [e_2,e_{2}]&=e_2,&\\
\alpha(e_{1})&=e_1,& \alpha(e_{2})&=0,&\beta(e_{1})&=0,& \beta(e_{2})&=e_2.&   	
\end{align*}		
\begin{align*}
L_{1}^{10}:[e_{1},e_{1}]&=0,&     [e_{1},e_{2}]&=e_1+e_2,
&
[e_2,e_{1}]&=-e_1-e_2,
&       [e_2,e_{2}]&=0,&\\
\alpha(e_{1})&=e_{1},& \alpha(e_{2})&=e_2,&\beta(e_{1})&=e_{1},& \beta(e_{2})&=e_{2}.&   	
\end{align*}
		\begin{align*}
		L_{1}^{11}:[e_{1},e_{1}]&=0,&     [e_{1},e_{2}]&=0,
		&
		[e_2,e_{1}]&=e_{1}
		,&       [e_2,e_{2}]&=e_{1}.&\\
		\alpha(e_{1})&=e_{1},& \alpha(e_{2})&=e_{2}\, ,&
		\beta(e_{1})&=0,& \beta(e_{2})&=ze_{1}.&  	
		\end{align*}
				%%%%%%%%%%%%%%%%%%%%%%Alg 3 cas 3 z1\neq0; t1=0	
		\begin{align*}
		L_{2}^{11}:[e_{1},e_{1}]&=0,&     [e_{1},e_{2}]&=0,
		&
		[e_2,e_{1}]&=e_{1}
		,&       [e_2,e_{2}]&=0.&\\
		\alpha(e_{1})&=e_{1},& \alpha(e_{2})&=e_{2}\, ,&
		\beta(e_{1})&=0,& \beta(e_{2})&=e_{1}.&  	
		\end{align*}
		%%%%%%%%%%%%%%%%%%%%%%Alg 3 cas 3 z1=0; t1\neq 0	
\begin{align*}
L_{3}^{11}:[e_{1},e_{1}]&=0,&    [e_{1},e_{2}]&=0,
&
[e_2,e_{1}]&=0
,&       [e_2,e_{2}]&=e_{1}.&\\
\alpha_{34}(e_{1})&=e_{1},& \alpha(e_{2})&=e_{2}\, ,&
\beta(e_{1})&=0,& \beta(e_{2})&=e_{1}.&  	
\end{align*}				
		%%%%%%%%%%%%%%%%%%%%%%Alg 7 cas 3
		\begin{align*}
		L_{1}^{12}:[e_{1},e_{1}]&=0,&     [e_{1},e_{2}]&=e_{1},
		&
		[e_2,e_{1}]&=-e_{1}
		,&       [e_2,e_{2}]&=-e_1.&\\
		\alpha(e_{1})&=e_{1},& \alpha(e_{2})&=e_{2}\, ,&
		\beta(e_{1})&=e_{1},& \beta(e_{2})&=e_{1}+e_{2}.&   	
		\end{align*}
		%%%%%%%%%%%%%%%%%%%%%%Alg 2 cas 4 y1\neq 0
		\begin{align*}
		L_{1}^{13}:[e_{1},e_{1}]&=0,&     [e_{1},e_{2}]&=e_{1},
		&
		[e_2,e_{1}]&=z_1e_{1}
		,&       [e_2,e_{2}]&=t_1e_{1},&\\
		\alpha(e_{1})&=0,& \alpha(e_{2})&=e_{1}\, ,&
		\beta(e_{1})&=0,& \beta(e_{2})&=e_{1}.&   	
		\end{align*}
	%%%%%%%%%%%%%%%%%%%%%%Alg 2 cas 4,y1=0 z1\neq 0
\begin{align*}
L_{2}^{13}:[e_{1},e_{1}]&=0,&     [e_{1},e_{2}]&=0,
&
[e_2,e_{1}]&=e_{1}
,&       [e_2,e_{2}]&=t_1e_{1},&\\
\alpha(e_{1})&=0,& \alpha(e_{2})&=e_{1}\, ,&
\beta(e_{1})&=0,& \beta(e_{2})&=ze_{1}.&   	
\end{align*}		
	%%%%%%%%%%%%%%%%%%%%%%Alg 2 cas 4 y1\neq 0
\begin{align*}
L_{3}^{13}:[e_{1},e_{1}]&=0,&     [e_{1},e_{2}]&=0,
&
[e_2,e_{1}]&=0
,&       [e_2,e_{2}]&=e_{1},&\\
\alpha(e_{1})&=0,& \alpha(e_{2})&=e_{1}\, ,&
\beta(e_{1})&=0,& \beta(e_{2})&=ze_{1}.&   	
\end{align*}		
%==================================			
		%%%%%%%%%%%%%%%%%%%%%%Alg 2 cas 4
	%	\begin{align*}
	%	L_{24}:[e_{1},e_{1}]_{36}&=0&,     [e_{1},e_{2}]_{36}&=0,
	%	&
	%	[e_2,e_{1}]_{36}&=e_{1}
	%	,&       [e_2,e_{2}]_{36}&=t_1e_{1},&\\
	%	\alpha_{36}(e_{1})&=e_{1},& \alpha_{36}(e_{2})&=e_{1}\, ,&
	%	\beta_{36}(e_{1})&=0,& \beta_{36}(e_{2})&=ze_{1}.&   	
	%	\end{align*}
		%%%%%%%%%%%%%%%%%%%%%%Alg 12 cas 4
		\begin{align*}
		L_{1}^{14}:[e_{1},e_{1}]&=0,&     [e_{1},e_{2}]&=0,
		&
		[e_2,e_{1}]&=0
		,&       [e_2,e_{2}]&=e_{1},&\\
		\alpha(e_{1})&=e_{1},& \alpha(e_{2})&=e_{1}+e_{2}\, ,&
		\beta(e_{1})&=0,& \beta(e_{2})&=ze_{1}.&   	
		\end{align*}
		%%%%%%%%%%%%%%%%%%%%%%Alg 17 cas 4,y1\neq 0	
		\begin{align*}
		L_{1}^{15}:[e_{1},e_{1}]&=0,&     [e_{1},e_{2}]&=e_{1},
		&
		[e_2,e_{1}]&=0
		,&       [e_2,e_{2}]&=t_{1}e_{1},&\\
		\alpha(e_{1})&=0,& \alpha(e_{2})&=e_{1}\, ,&
		\beta(e_{1})&=e_{1},& \beta(e_{2})&=e_{2}.&    	
		\end{align*}
		%%%%%%%%%%%%%%%%%%%%%%Alg 17 cas 4 y1=0	
		%\begin{align*}
	%	L_{2}^{22}:[e_{1},e_{1}]&=0&,     [e_{1},e_{2}]_{37}&=0,
	%	&
	%	[e_2,e_{1}]&=0
	%	,&       [e_2,e_{2}]_{37}&=e_{1},&\\
	%	\alpha(e_{1})&=0,& \alpha(e_{2})&=e_{1}\, ,&
	%	\beta(e_{1})&=e_{2},& \beta(e_{2})&=e_{2}.&    	
	%	\end{align*}
	%	%%%%%%%%%%%%%%%%%%%%%%Alg 17 cas 4	
	%	\begin{align*}
	%	L_{28}:[e_{1},e_{1}]_{38}&=0&,     [e_{1},e_{2}]_{38}&=e_{1},
	%	&
	%	[e_2,e_{1}]_{38}&=0
	%	,&       [e_2,e_{2}]_{38}&=0,&\\
	%	\alpha_{38}(e_{1})&=0,& \alpha_{38}(e_{2})&=e_{1}\, ,&
	%	\beta_{38}(e_{1})&=e_{1},& \beta_{38}(e_{2})&=e_{2}.&    	
	%	\end{align*}
	%	\begin{align*}
	%	L_{29}:[e_{1},e_{1}]_{39}&=0&,     [e_{1},e_{2}]_{39}&=0,
	%	&
	%	[e_2,e_{1}]_{39}&=0
	%	,&       [e_2,e_{2}]_{39}&=e_{1},&\\
	%	\alpha_{39}(e_{1})&=0,& \alpha_{39}(e_{2})&=e_{1}\, ,&
	%	\beta_{39}(e_{1})&=e_{1},& \beta_{39}(e_{2})&=e_{2}.&    	
	%	\end{align*}
		%%%%%%%%%%%%%%%%%%%%%%Alg 18 cas 4	
		\begin{align*}
		L_{1}^{16}:[e_{1},e_{1}]&=0,&     [e_{1},e_{2}]&=0,
		&
		[e_2,e_{1}]&=0
		,&       [e_2,e_{2}]&=e_{1},& \\
		\alpha(e_{1})&=0,& \alpha(e_{2})&=e_{1}\, ,&
		\beta(e_{1})&=e_{1},& \beta(e_{2})&=ze_{1}+e_{2}.&   	
		\end{align*}
		%%%%%%%%%%%%%%%%%%%%%%Alg 19 cas 4	
		\begin{align*}
		L_{1}^{17}: [e_{1},e_{1}]&=0, & [e_{1},e_{2}]&=e_{1}, & [e_2,e_{1}]&=-e_{1}, & [e_2,e_{2}]&=(1-z)e_{1},\\
		\alpha(e_{1})&=e_{1},& \alpha(e_{2})&=e_{1}+e_{2}, & \beta(e_{1})&=e_{1},& \beta(e_{2})&=ze_{1}+e_{2}.   	
		\end{align*}
%	\end{enumerate}
	%%%%%%%%%%%%%%%%%%%%%%%%%%%%%%%%%%%%%%	
\end{prop}

\begin{cor}
		Every  decomposable $ 2 $-dimensional multiplicative BiHom-Lie algebra $ L $  is   isomorphic to one of these $ 4 $ algebras:
	$L_{3}^{1}  $,
	$ L_{1}^{2} $,
	$L_{1}^{5} $,
	$ L_{1}^{9} $.
\end{cor}

\begin{remq} $ <e_1+e_2> $ is an ideal of BiHom-Lie algebra $ L_{1}^{10} $. For the others BiHom-Lie algebras $ <e_1> $ is an ideal of $ L_{i}^{j} $. Hence,
	every   $ 2 $-dimensional multiplicative BiHom-Lie algebra is not simple.
\end{remq}

\section{Centroids and derivations of  $ 2 $-dimensional multiplicative BiHom-Lie algebras}
\label{sec:centrderiv2dimmultipBiHomLiealg}
Let  $(L,[\cdot,\cdot ],\alpha,\beta)$ be a $ n $-dimensional multiplicative  BiHom-Lie  algebra. Let $$\displaystyle \alpha^{r}\beta^{l}(e_{j})=\sum_{k=1}^{n}m_{kj}e_{k}. $$
An element $ d $ of $ Der^{(\delta,\mu,\gamma)}_{\alpha^{r}\beta^{l}}(L) $, being a linear transformation of the vector space $ L $,	is represented in a matrix form $ (d_{ij})_{1\leq i,j\leq n} $ corresponding to $\displaystyle d(e_{j})=\sum_{k=1}^{n}d_{kj}e_{k} $, for $ j=1,\dots,n $. According to the definition of the $ (\delta,\mu,\gamma) $-$ \alpha^{r}\beta^{l} $-derivation the entries $ d_{ij} $ of the matrix $ (d_{ij})_{1\leq i,j\leq n} $
must satisfy the following systems $ \mathcal{S} $ of equations:
\begin{align*}
 \displaystyle \sum_{k=1}^{n} d_{ik}a_{kj} &= \sum_{k=1}^{n} a_{ik}d_{kj};\qquad
  \displaystyle \sum_{k=1}^{n} d_{ik}b_{kj} = \sum_{k=1}^{n} b_{ik}d_{kj};\\
\displaystyle \delta\sum_{k=1}^{n} c_{ij}^{k}d_{sk}&-\mu \sum_{k=1}^{n} \sum_{l=1}^{n}d_{ki}m_{lj}c_{kl}^{s} -\gamma   \sum_{k=1}^{n} \sum_{l=1}^{n}d_{lj} m_{ki}c_{kl}^{s}=0,
\end{align*}
where $ (a_{ij})_{1\leq i,j\leq n} $ is the matrix of $ \alpha $, $ (b_{ij})_{1\leq i,j\leq n} $ is the matrix of $ \beta $ and
$( c_{ij}^{k}) $ are the structure constants of $ L $.
%\subsection{Centroids  of  $ 2 $-dimensional multiplicative BiHom-Lie algebras}	
First, let us give the following definitions:
\begin{defn}
A BiHom-Lie algebra  is called characteristically nilpotent (denoted by CN) if the Lie algebra $ Der_{\alpha^{0}\beta^{0}}(L) $ is nilpotent.
\end{defn}
\begin{defn}
Let $ L $ be an indecomposable 	BiHom-Lie algebra. We say $ L $
is small if $\varGamma_{\alpha^{0}\beta^{0}}(L)  $ is generated by central derivation and the scalars.	
The centroid of a decomposable 	BiHom-Lie algebra is small if the centroids of each indecomposable factor are
small.	
\end{defn}
Now we apply the algorithms mention in the previous paragraph to centroid and derivation of $ 2 $-dimensional complex BiHom-Lie
algebras. To find the centroids and derivations of $ 2 $-dimensional complex   BiHom-Lie
algebras  we use the classification results from the previous section. The results are given in the following theorem. Moreover, we give the type of
$ \varGamma_{\alpha^r\beta^l}(L_{i}^{j}) $ and $ Der_{\alpha^r\beta^l}(L_{i}^{j})  $ if $ (r,l)=(0,0) $.
\begin{theorem}
\begin{align*}
L_{1}^{1}:[e_{1},e_{1}]&=e_1,&     [e_{1},e_{2}]&=e_{1},
&
[e_2,e_{1}]&=z_1 e_{1}
,&       [e_2,e_{2}]&=0,&\\
\alpha(e_{1})&=0,& \alpha(e_{2})&=be_{2},&\beta(e_{1})&=0,& \beta(e_{2})&=ye_{2}.&   	
\end{align*}
\begin{center}
	\begin{tabular}{|c|c|c|c|c|c|}	
	\hline
  $ \alpha^r\beta^l $ &&$ \varGamma_{\alpha^r\beta^l}(L_{1}^{1}) $&
  \begin{minipage}{2cm} \vspace{0,3cm} Type of \\ $ \varGamma_{\alpha^0\beta^0}(L_{1}^{1}) $  \\ \end{minipage}  &$ Der_{\alpha^r\beta^l}(L_{1}^{1}) $&CN\\
	\hline
$ (r,l) =(0,0)$&$ z_1=0 $&$ \begin{pmatrix}
c_1&0\\
0&c_2
\end{pmatrix} $	&Not small   &$ \begin{pmatrix}
0&0\\
0&0
\end{pmatrix} $& Yes\\
	\hline
$ (r,l) =(0,0)$&$ z_1\neq0 $&$ \begin{pmatrix}
c_1&0\\
0&c_1
\end{pmatrix} $	&Small    &$ \begin{pmatrix}
0&0\\
0&0
\end{pmatrix} $ &Yes\\

	\hline
$ (r,l) \neq(0,0)$&&$ \begin{pmatrix}
0&0\\
0&c_2
\end{pmatrix} $	&& $ \begin{pmatrix}
0&0\\
0&d_2
\end{pmatrix} $& \\
	\hline
\end{tabular}
\end{center}
%%%%%%%%%%%%%%%%%%%%%				
\begin{align*}
L_{2}^{1}:[e_{1},e_{1}]&=e_1,&     [e_{1},e_{2}]&=0,
&
[e_2,e_{1}]&= e_{1}
,&       [e_2,e_{2}]&=0,&\\
\alpha(e_{1})&=0,& \alpha(e_{2})&=be_{2},&\beta(e_{1})&=0,& \beta(e_{2})&=ye_{2}.&   	
\end{align*}
\begin{center}
	\begin{tabular}{|c|c|c|c|c|c|}	
		\hline
	$ \alpha^r\beta^l $	&$ \varGamma_{\alpha^r\beta^l}(L_{2}^{1}) $&
\begin{minipage}{2cm} \vspace{0,3cm} Type of \\ $ \varGamma_{\alpha^0\beta^0}(L_{1}^{1}) $  \\ \end{minipage} &$ Der_{\alpha^r\beta^l}(L_{2}^{1}) $&CN\\
		\hline
		$ (r,l) =(0,0)$&$ \begin{pmatrix}
		c_1&0\\
		0&c_1
		\end{pmatrix} $	&Small&$ \begin{pmatrix}
		0&0\\
		0&0
		\end{pmatrix} $& Yes\\
		\hline
		$ (r,l) \neq(0,0)$&$ \begin{pmatrix}
		0&0\\
		0&c_2
		\end{pmatrix} $	&& $ \begin{pmatrix}
		0&0\\
		0&d_2
		\end{pmatrix} $& \\
		\hline
	\end{tabular}
\end{center}
%%%%%%%%%%%%%%%%%%%%%				
\begin{align*}
L_{3}^{1}:[e_{1},e_{1}]&=e_1,&     [e_{1},e_{2}]&=0,
&
[e_2,e_{1}]&= 0
,&       [e_2,e_{2}]&=0,&\\
\alpha(e_{1})&=0,& \alpha(e_{2})&=be_{2},&\beta(e_{1})&=0,& \beta(e_{2})&=ye_{2}.&   	
\end{align*}
\begin{center}
	\begin{tabular}{|c|c|c|c|c|}	
		\hline
	$ \alpha^r\beta^l $	&$ \varGamma_{\alpha^r\beta^l}(L_{3}^{1}) $&
\begin{minipage}{2cm} \vspace{0,3cm} Type of \\ $ \varGamma_{\alpha^0\beta^0}(L_{3}^{1}) $ \\ \end{minipage} &$ Der_{\alpha^r\beta^l}(L_{3}^{1}) $&CN\\
	\hline		
%%%%%%%%%		\hline	
$ (r,l) =(0,0)$&$ \begin{pmatrix}
c_1&0\\
0&c_1
\end{pmatrix} $	& Not small&$ \begin{pmatrix}
0&0\\
0&d_2
\end{pmatrix} $&Yes \\
\hline
$ (r,l) \neq(0,0)$&$ \begin{pmatrix}
0&0\\
0&c_2
\end{pmatrix} $	& &$ \begin{pmatrix}
0&0\\
0&d_2
\end{pmatrix} $& \\
 		\hline
\end{tabular}
\end{center}
%%%%%%%%%%%%%
\begin{align*}
L_{4}^{1}:[e_{1},e_{1}]&=0,&     [e_{1},e_{2}]&=e_{1},
&
[e_2,e_{1}]&= z_1e_{1}
,&       [e_2,e_{2}]&=0,&\\
\alpha(e_{1})&=0,& \alpha(e_{2})&=be_{2},&\beta(e_{1})&=0,& \beta(e_{2})&=ye_{2}.&   	
\end{align*}
\begin{center}
	\begin{tabular}{|c|c|c|c|c|c|}	
		\hline
	&	&$ \varGamma_{\alpha^r\beta^l}(L_{4}^{1}) $&
\begin{minipage}{2cm} \vspace{0,3cm} Type of \\ $ \varGamma_{\alpha^0\beta^0}(L_{4}^{1}) $ \\ \end{minipage}
&$ Der_{\alpha^r\beta^l}(L_{4}^{1}) $&CN\\
		\hline		
		%%%%%%%%%		\hline	
		$ (r,l) =(0,0)$&$ z_1=0 $  &$ \begin{pmatrix}
		c_1&0\\
		0&c_2
		\end{pmatrix} $	&Not small&$ \begin{pmatrix}
		d_1&0\\
		0&0
		\end{pmatrix} $&Yes \\
		\hline
		%%%%%%%%%		\hline	
$ (r,l) =(0,0)$&$ z_1\neq 0 $  &$ \begin{pmatrix}
c_1&0\\
0&c_1
\end{pmatrix} $	&Small&$ \begin{pmatrix}
d_1&0\\
0&0
\end{pmatrix} $&Yes \\
\hline
		%%%%%%%%%		\hline	
$ (r,l) \neq (0,0)$ &\begin{minipage}{2cm} \begin{eqnarray*}  z_1 &=& 0; \\ b^ry^l &=& 1 \\ \end{eqnarray*} \end{minipage} &$ \begin{pmatrix}
c_1&0\\
0&c_2
\end{pmatrix} $	&&$ \begin{pmatrix}
d_1&0\\
0&d_2
\end{pmatrix} $& \\
\hline
		%%%%%%%%%		\hline	
$ (r,l) \neq (0,0)$&  $ b^ry^l\neq1 $ &$ \begin{pmatrix}
0&0\\
0&c_2
\end{pmatrix} $	& &$ \begin{pmatrix}
0&0\\
0&d_2
\end{pmatrix} $& \\
\hline
	%%%%%%%%%		\hline	
$ (r,l) \neq (0,0)$& \begin{minipage}{2cm} \begin{eqnarray*} z_1 &\neq & 0; \\ b^ry^l &=& 1 \\ \end{eqnarray*} \end{minipage} &$ \begin{pmatrix}
0&0\\
0&c_2
\end{pmatrix} $	&&$ \begin{pmatrix}
d_1&0\\
0&d_2
\end{pmatrix} $& \\
\hline
	%%%%%%%%%		\hline	
%$ (r,l) \neq (0,0)$&$ z_1\neq 0 $ and  $ b^ry^l\neq 1 $ &$ \begin{pmatrix}
%0&0\\
%0&c_2
%\end{pmatrix} $	&&$ \begin{pmatrix}
%0&0\\
%0&d_2
%\end{pmatrix} $& \\
%\hline
\end{tabular}
\end{center}
%%%%%%%%%%%%%

%============Alg 15 cas 1==============	
%%%%%%%%%%%%%%%%%%%%%				
\begin{align*}
L_{5}^{1}:[e_{1},e_{1}]&=0,&     [e_{1},e_{2}]&=0,
&
[e_2,e_{1}]&= e_{1}
,&       [e_2,e_{2}]&=0,&\\
\alpha(e_{1})&=0,& \alpha(e_{2})&=be_{2},&\beta(e_{1})&=0,& \beta(e_{2})&=ye_{2}.&   	
\end{align*}
\begin{center}
	\begin{tabular}{|c|c|c|c|c|c|c|c|}	
		\hline
		&	&$ \varGamma_{\alpha^r\beta^l}(L_{5}^{1}) $&
\begin{minipage}{2cm} \vspace{0,3cm} Type of \\ $ \varGamma_{\alpha^0\beta^0}(L_{5}^{1}) $ \\ \end{minipage} &$ Der_{\alpha^r\beta^l}(L_{5}^{1}) $&CN\\
		\hline		
		%%%%%%%%%
		$ (r,l) =(0,0)$&  &$ \begin{pmatrix}
		c_1&0\\
		0&c_1
		\end{pmatrix} $	&Small&$ \begin{pmatrix}
		d_1&0\\
		0&0
		\end{pmatrix} $&Yes \\
		\hline		
%%%%%%%%%%%%%			
$ (r,l) \neq (0,0)$& $ b^ry^l=1 $ &$ \begin{pmatrix}
0&0\\
0&c_2
\end{pmatrix} $	&&$ \begin{pmatrix}
d_1&0\\
0&d_2
\end{pmatrix} $& \\
\hline
%%%%%%%%%%%%%
$ (r,l) \neq (0,0)$& $ b^ry^l\neq1 $ &$ \begin{pmatrix}
0&0\\
0&c_2
\end{pmatrix} $	&&$ \begin{pmatrix}
0&0\\
0&d_2
\end{pmatrix} $& \\
\hline
\end{tabular}
\end{center}
%%%%%%%%%%%%%

%============Alg 65 cas 1==============	
%%%%%%%%%%%%%%%%%%%%%
\begin{align*}
L_{1}^{2}:[e_{1},e_{1}]&= e_{1},&     [e_{1},e_{2}]&=0,
&
[e_2,e_{1}]&=0
,&       [e_2,e_{2}]&=0,&\\
\alpha(e_{1})&= e_{1},& \alpha(e_{2})&=be_{2},&\beta(e_{1})&=0,& \beta(e_{2})&=ye_{2}.&   	
\end{align*}
%%%%%%%%%%%%%%%%%%%%%%%%%%%%%%%%%
\begin{center}
	\begin{tabular}{|c|c|c|c|c|c|c|}	
		\hline
			&$ \varGamma_{\alpha^r\beta^l}(L_{1}^{2}) $ &
\begin{minipage}{2cm} \vspace{0,3cm} Type of \\ $ \varGamma_{\alpha^0\beta^0}(L_{1}^{2}) $ \\ \end{minipage} &$ Der_{\alpha^r\beta^l}(L_{1}^{2}) $&CN\\
		\hline		
		%%%%%%%%%
%%%%%%%%%%%%%
$ l=0$& $ \begin{pmatrix}
c_1&0\\
0&c_2
\end{pmatrix} $	&Not small&$ \begin{pmatrix}
0&0\\
0&d_2
\end{pmatrix} $&Yes \\
\hline
%%%%%%%%%%%%%
$ l\neq0$&  $ \begin{pmatrix}
0&0\\
0&c_2
\end{pmatrix} $	&&$ \begin{pmatrix}
0&0\\
0&d_2
\end{pmatrix} $& \\
\hline
\end{tabular}
\end{center}
%%%%%%%%%%%%%
%============Alg 66 cas 1==============y_1\neq 0	
%%%%%%%%%%%%%%%%%%%%%
%\begin{align*}
%L_{1}^{3}:[e_{1},e_{1}]&=0 ,&     [e_{1},e_{2}]&=e_{1},
%&
%[e_2,e_{1}]&=z_{1}e_{1}
%,&       [e_2,e_{2}]&=0,&\\
%\alpha(e_{1})&= 0,& \alpha(e_{2})&=e_{2},&\beta(e_{1})&=0,& \beta(e_{2})&=ye_{2}.&   	
%\end{align*}
%============Alg 66 cas 1==============y_1= 0	
%%%%%%%%%%%%%%%%%%%%%
%\begin{align*}
%L_{2}^{3}:[e_{1},e_{1}]&=0 ,&     [e_{1},e_{2}]&=0,
%&
%[e_2,e_{1}]&=e_{1}
%,&       [e_2,e_{2}]&=0,&\\
%\alpha(e_{1})&= 0,& \alpha(e_{2})&=e_{2},&\beta(e_{1})&=0,& \beta(e_{2})&=ye_{2}.&   	
%\end{align*} 		
%============Alg 68 cas 1==============	
%%%%%%%%%%%%%%%%%%%%%
\begin{align*}
L_{1}^{3}:[e_{1},e_{1}]&=0 ,&     [e_{1},e_{2}]&=e_{1},
&
[e_2,e_{1}]&=0
,&       [e_2,e_{2}]&=0,&\\
\alpha(e_{1})&=a e_{1},& \alpha(e_{2})&=e_{2},&\beta(e_{1})&=0,& \beta(e_{2})&=ye_{2}.&   	
\end{align*}
%%%%%%%%%%%%%%%%%%%%%%%%%%%%%%%%%
\begin{center}
	\begin{tabular}{|c|c|c|c|c|c|c|}	
		\hline
	&	&$ \varGamma_{\alpha^r\beta^l}(L_{1}^{3}) $&
\begin{minipage}{2cm} \vspace{0,3cm} Type of \\ $ \varGamma_{\alpha^0\beta^0}(L_{1}^{3}) $ \\ \end{minipage} &$ Der_{\alpha^r\beta^l}(L_{1}^{3}) $&CN\\
		\hline		
		%%%%%%%%%
		%%%%%%%%%%%%%
		$ l=0$&     &$ \begin{pmatrix}
		c_1&0\\
		0&c_2
		\end{pmatrix} $	&Not small&$ \begin{pmatrix}
		d_1&0\\
		0&0
		\end{pmatrix} $&Yes \\
		\hline
		%%%%%%%%%%%%%
		$ l\neq0$& $ y^l=1 $ &$ \begin{pmatrix}
		c_1&0\\
		0&c_2
		\end{pmatrix} $	&&$ \begin{pmatrix}
		d_1&0\\
		0&d_2
		\end{pmatrix} $& \\
		\hline
		%%%%%%%%%%%%%
$ l\neq0$& $ y^l\neq 1 $ &$ \begin{pmatrix}
0&0\\
0&c_2
\end{pmatrix} $&	&$ \begin{pmatrix}
0&0\\
0&d_2
\end{pmatrix} $& \\
\hline		
	\end{tabular}
\end{center}
%%%%%%%%%%%%%

%============Alg 67 cas 1==============x1\neq	 0;y1\neq	 0;
%%%%%%%%%%%%%%%%%%%%%

\begin{align*}
L_{1}^{4}:[e_{1},e_{1}]&=e_1,&     [e_{1},e_{2}]&=e_{1},
&
[e_2,e_{1}]&=0
,&       [e_2,e_{2}]&=0,&\\
\alpha(e_{1})&=e_{1},& \alpha(e_{2})&=e_{2},&\beta(e_{1})&=0,& \beta(e_{2})&=ye_{2}.&   	
\end{align*}
%%%%%%%%%%%%%%%%%%%%%%%%%%%%%%%%%
\begin{center}
	\begin{tabular}{|c|c|c|c|c|c|c|}	
		\hline
			&$ \varGamma_{\alpha^r\beta^l}(L_{1}^{4}) $&
\begin{minipage}{2cm} \vspace{0,3cm} Type of \\ $ \varGamma_{\alpha^0\beta^0}(L_{1}^{4}) $ \\ \end{minipage} &$ Der_{\alpha^r\beta^l}(L_{1}^{4}) $&CN\\
		\hline
	$ l=0$     &$ \begin{pmatrix}
c_1&0\\
0&c_2
\end{pmatrix} $	&Not small&$ \begin{pmatrix}
0&0\\
0&0
\end{pmatrix} $& Yes\\
\hline
%%%%%%%%%%%%%
		%%%%%%%%%%%%%
$ l\neq0$  &$ \begin{pmatrix}
0&0\\
0&c_2
\end{pmatrix} $	&&$ \begin{pmatrix}
0&0\\
0&d_2
\end{pmatrix} $ &\\
\hline

	\end{tabular}
\end{center}

%%%%%%%%%%%%%%%%%%%%%
%============Alg 67 cas 1==============x1\neq	 0;y1=	 0;
%%%%%%%%%%%%%%%%%%%%%
%\begin{align*}
%L_{2}^{5}:[e_{1},e_{1}]&=e_1,&     [e_{1},e_{2}]&=0,
%&
%[e_2,e_{1}]&=0
%,&       [e_2,e_{2}]&=0,&\\
%\alpha(e_{1})&=e_{1},& \alpha(e_{2})&=e_{2},&\beta(e_{1})&=0,& \beta(e_{2})&=ye_{2}.&   	
%\end{align*}
%%%%%%%%%%%%%%%%%%%%%
%============Alg 67 cas 1==============x1=	 0
%%%%%%%%%%%%%%%%%%%%%
%\begin{align*}
%L_{3}^{5}:[e_{1},e_{1}]&=0,&     [e_{1},e_{2}]&=e_1,
%&
%[e_2,e_{1}]&=0
%,&       [e_2,e_{2}]&=0,&\\
%\alpha(e_{1})&=e_{1},& \alpha(e_{2})&=e_{2},&\beta(e_{1})&=0,& \beta(e_{2})&=ye_{2}.&   	
%\end{align*}
%%%%%%%%%%%%%%%%%%%%% 		
%============Alg 73 cas 1==============y1\neq	 0
%%%%%%%%%%%%%%%%%%%%%
%\begin{align*}
%L_{1}^{6}:[e_{1},e_{1}]&=0,&     [e_{1},e_{2}]&=e_{1},
%&
%[e_2,e_{1}]&=z_{1}e_{1}
%,&       [e_2,e_{2}]&=0,&\\
%\alpha(e_{1})&=0,& \alpha(e_{2})&=be_{2},&\beta(e_{1})&=0,& \beta(e_{2})&=e_{2}.&   	
%\end{align*}
%============Alg 73 cas 1==============y1=	 0
%%%%%%%%%%%%%%%%%%%%%
%\begin{align*}
%L_{2}^{6}:[e_{1},e_{1}]&=0,&     [e_{1},e_{2}]&=0,
%&
%[e_2,e_{1}]&=e_{1}
%,&       [e_2,e_{2}]&=0,&\\
%\alpha(e_{1})&=0,& \alpha(e_{2})&=be_{2},&\beta(e_{1})&=0,& \beta(e_{2})&=e_{2}.&   	
%\end{align*}		
%%%%%%%%%%%%%  Alg 76  %%%%%%%%%%%%%%
\begin{align*}
L_{1}^{5}:[e_{1},e_{1}]&=e_{1},&     [e_{1},e_{2}]&=0,
&
[e_2,e_{1}]&=0
,&       [e_2,e_{2}]&=0,&\\
\alpha(e_{1})&=0,& \alpha(e_{2})&=be_{2},&\beta(e_{1})&=e_{1},& \beta(e_{2})&=ye_{2}.&   	
\end{align*}
%%%%%%%%%%%%%%%%%%%%%%%%%%%%%%%%%
\begin{center}
	\begin{tabular}{|c|c|c|c|c|c|c|}	
		\hline
		&$ \varGamma_{\alpha^r\beta^l}(L_{1}^{5}) $&
\begin{minipage}{2cm} \vspace{0,3cm} Type of \\ $ \varGamma_{\alpha^0\beta^0}(L_{1}^{5}) $ \\ \end{minipage} &$ Der_{\alpha^r\beta^l}(L_{1}^{5}) $&CN\\
		\hline
		$ r=0$     &$ \begin{pmatrix}
		c_1&0\\
		0&c_2
		\end{pmatrix} $	&Not small&$ \begin{pmatrix}
		0&0\\
		0&d_2
		\end{pmatrix} $& \\
		\hline
		%%%%%%%%%%%%%
		%%%%%%%%%%%%%
		$ r\neq0$  &$ \begin{pmatrix}
		0&0\\
		0&c_2
		\end{pmatrix} $	&&$ \begin{pmatrix}
		0&0\\
		0&d_2
		\end{pmatrix} $ &\\
		\hline	
	\end{tabular}
\end{center}	
%%%%%%%%%%%%%  Alg 80  %%%%%%%%%%%%%%
\begin{align*}
L_{1}^{6}:[e_{1},e_{1}]&=e_{1},&     [e_{1},e_{2}]&=0,
&
[e_2,e_{1}]&=e_{1}
,&       [e_2,e_{2}]&=0,&\\
\alpha(e_{1})&=0,& \alpha(e_{2})&=be_{2},&\beta(e_{1})&=e_{1},& \beta(e_{2})&=e_{2}.&   	
\end{align*}
%%%%%%%%%%%%%%%%%%%%%%%%%%%%%%%%%
\begin{center}
	\begin{tabular}{|c|c|c|c|c|c|c|}	
		\hline
		&$ \varGamma_{\alpha^r\beta^l}(L_{1}^{6}) $&
\begin{minipage}{2cm} \vspace{0,3cm} Type of \\ $ \varGamma_{\alpha^0\beta^0}(L_{1}^{6}) $ \\ \end{minipage} &$ Der_{\alpha^r\beta^l}(L_{1}^{6}) $&Yes\\
		\hline
		$ r=0$     &$ \begin{pmatrix}
		c_1&0\\
		0&c_1
		\end{pmatrix} $	&Small&$ \begin{pmatrix}
		0&0\\
		0&d_2
		\end{pmatrix} $&Yes \\
		\hline
		%%%%%%%%%%%%%
		%%%%%%%%%%%%%
		$ r\neq0$  &$ \begin{pmatrix}
		0&0\\
		0&c_2
		\end{pmatrix} $	&&$ \begin{pmatrix}
		0&0\\
		0&0
		\end{pmatrix} $& \\
		\hline	
	\end{tabular}
\end{center}
%\begin{align*}
%L_{2}^{8}:[e_{1},e_{1}]&=e_{1},&     [e_{1},e_{2}]&=0,
%&
%[e_2,e_{1}]&=0
%,&       [e_2,e_{2}]&=0,&\\
%\alpha(e_{1})&=0,& \alpha(e_{2})&=be_{2},&\beta(e_{1})&=e_{1},& \beta(e_{2})&=e_{2}.&   	
%\end{align*}
%\begin{align*}
%L_{3}^{8}:[e_{1},e_{1}]&=0,&     [e_{1},e_{2}]&=0,
%&
%[e_2,e_{1}]&=e_{1}
%,&       [e_2,e_{2}]&=0,&\\
%\alpha(e_{1})&=0,& \alpha(e_{2})&=be_{2},&\beta(e_{1})&=e_{1},& \beta(e_{2})&=e_{2}.&   	
%\end{align*}		
%%%%%%%%%%%%%  Alg 93  %%%%%%%%%%%%%%
\begin{align*}
L_{1}^{7}:[e_{1},e_{1}]&=0,&     [e_{1},e_{2}]&=0,
&
[e_2,e_{1}]&=e_1
,&       [e_2,e_{2}]&=0,&\\
\alpha(e_{1})&=0,& \alpha(e_{2})&=be_2,&\beta(e_{1})&=xe_{1},\, x\neq 0& \beta(e_{2})&=e_2.&   	
\end{align*}
%%%%%%%%%%%%%%%%%%%%%%%%%%%%%%%%%
\begin{center}
	\begin{tabular}{|c|c|c|c|c|c|c|}	
		\hline
	&	&$ \varGamma_{\alpha^r\beta^l}(L_{1}^{7}) $&
\begin{minipage}{2cm} \vspace{0,3cm} Type of \\ $ \varGamma_{\alpha^0\beta^0}(L_{1}^{7}) $ \\ \end{minipage} &$ Der_{\alpha^r\beta^l}(L_{1}^{7}) $&CN\\
	%	\hline
	%	$ r=0$ &$ x=0 $    &$ \begin{pmatrix}
		%0&0\\
	%	0&c_2
		%\end{pmatrix} $	&&$ \begin{pmatrix}
		%d_1&0\\
	%	0&d_2
	%	\end{pmatrix} $& \\
		\hline
		%%%%%%%%%%%%%
		$ r=0$ &&$ \begin{pmatrix}
c_1&0\\
0&\frac{c_1}{x^l}
\end{pmatrix} $	&Small&$ \begin{pmatrix}
d_1&0\\
0&0
\end{pmatrix} $&Yes \\
\hline
%%%%%%%%%%%%%		
		%%%%%%%%%%%%%
		$ r\neq0$  & $ x=1 $  &$ \begin{pmatrix}
		0&0\\
		0&c_2
		\end{pmatrix} $	&&$ \begin{pmatrix}
		d_1&0\\
		0&d_2
		\end{pmatrix} $& \\
		\hline
		%%%%%%%%%%%%%
$ r\neq0$  &$ x\neq 1 $& $ \begin{pmatrix}
0&0\\
0&c_2
\end{pmatrix} $	&&$ \begin{pmatrix}
0&0\\
0&d_2
\end{pmatrix} $ &\\
\hline			
	\end{tabular}
\end{center}		
%%%%%%%%%%%%%  Alg 96  %%%%%%%%%%%%%%
\begin{align*}
L_{1}^{8}:[e_{1},e_{1}]&=0,&     [e_{1},e_{2}]&=e_1,
&
[e_2,e_{1}]&=-\frac{x}{a}e_1
,&       [e_2,e_{2}]&=0,&\\
\alpha(e_{1})&=ae_1,& \alpha(e_{2})&=e_2,&\beta(e_{1})&=xe_{1},& \beta(e_{2})&=e_2.&   	
\end{align*}
\begin{center}
	\begin{tabular}{|c|c|c|c|}	
		\hline
			$ \varGamma_{\alpha^r\beta^l}(L_{1}^{8}) $&
\begin{minipage}{2cm} \vspace{0,3cm} Type of \\ $ \varGamma_{\alpha^0\beta^0}(L_{1}^{8}) $ \\ \end{minipage} &$ Der_{\alpha^r\beta^l}(L_{1}^{8}) $&CN\\
		\hline
 $ \begin{pmatrix}
c_1&0\\
0&\frac{c_1}{a^rx^l}
\end{pmatrix} $	&Small&$ \begin{pmatrix}
d_1&0\\
0&0
\end{pmatrix} $&Yes \\*[0,3cm]
\hline
	\end{tabular}
\end{center}
%%%%%%%%%%%%%  Alg 101  %%%%%%%%%%%%%%
%\begin{align*}
%L_{1}^{9}:[e_{1},e_{1}]&=e_{1},&     [e_{1},e_{2}]&=0,
%&
%[e_2,e_{1}]&=0
%,&       [e_2,e_{2}]&=0,&\\
%\alpha(e_{1})&=e_1,& \alpha(e_{2})&=0,&\beta(e_{1})&=0,& \beta(e_{2})&=ye_2.&   	
%\end{align*}		
%%%%%%%%%%%%%  Alg 108 %%%%%%%%%%%%%% x1\neq0;t2\neq0
\begin{align*}
L_{1}^{9}:[e_{1},e_{1}]&=e_1,&     [e_{1},e_{2}]&=0,
&
[e_2,e_{1}]&=0
,&       [e_2,e_{2}]&=e_2,&\\
\alpha(e_{1})&=e_1,& \alpha(e_{2})&=0,&\beta(e_{1})&=0,& \beta(e_{2})&=e_2.&   	
\end{align*}
%%%%%%%%%%%%%%%%%%%%%%%%%%%%%%%%%
\begin{center}
	\begin{tabular}{|c|c|c|c|c|}	
	\hline
	&	$ \varGamma_{\alpha^r\beta^l}(L_{1}^{9}) $&
\begin{minipage}{2cm} \vspace{0,3cm} Type of \\ $ \varGamma_{\alpha^0\beta^0}(L_{1}^{9}) $ \\ \end{minipage} &$ Der_{\alpha^r\beta^l}(L_{1}^{9}) $&\\
	\hline
	%%%%%%%%%%%%%%%%
	$ r=0,\,l=0$ &  $ \begin{pmatrix}
	c_1&0\\
	0&c_1
	\end{pmatrix} $	&Not small&$ \begin{pmatrix}
	0&0\\
	0&0
	\end{pmatrix} $&Yes \\
	\hline
	%%%%%%%%%%%%%%%%
$ r=0, \, l\neq 0 $    &$ \begin{pmatrix}
0&0\\
0&0
\end{pmatrix} $	&&$ \begin{pmatrix}
0&0\\
0&0
\end{pmatrix} $& \\
\hline
	%%%%%%%%%%%%%%%%
$ r\neq0 $    &$ \begin{pmatrix}
0&0\\
0&c_2
\end{pmatrix} $	&&$ \begin{pmatrix}
0&0\\
0&d_2
\end{pmatrix} $& \\
\hline
	\end{tabular}
\end{center}

\begin{align*}
L_{1}^{10}:[e_{1},e_{1}]&=0,&     [e_{1},e_{2}]&=e_1+e_2,
&
[e_2,e_{1}]&=-e_1-e_2,
&       [e_2,e_{2}]&=0,&\\
\alpha(e_{1})&=e_{1},& \alpha(e_{2})&=e_2,&\beta(e_{1})&=e_{1},& \beta(e_{2})&=e_{2}.&   	
\end{align*}
%%%%%%%%%%%%%%%%%%%%%%%%%%%%%%%%%
\begin{center}
	\begin{tabular}{|c|c|c|c|}	
		\hline
			$ \varGamma_{\alpha^r\beta^l}(L_{1}^{10}) $&
\begin{minipage}{2cm} \vspace{0,3cm} Type of \\ $ \varGamma_{\alpha^0\beta^0}(L_{1}^{10}) $ \\ \end{minipage} &$ Der_{\alpha^r\beta^l}(L_{1}^{10}) $&CN\\
		\hline
		%%%%%%%%%%%%%%%%
    $ \begin{pmatrix}
c_1&0\\
0&c_1
\end{pmatrix} $	&Small&$ \begin{pmatrix}
d_1&d_2\\
d_1&d_2
\end{pmatrix} $&No \\
\hline
\end{tabular}
\end{center}

\begin{align*}
L_{1}^{11}:[e_{1},e_{1}]&=0,&     [e_{1},e_{2}]&=0,
&
[e_2,e_{1}]&=e_{1}
,&       [e_2,e_{2}]&=e_{1}.&\\
\alpha(e_{1})&=e_{1},& \alpha(e_{2})&=e_{2}\, ,&
\beta(e_{1})&=0,& \beta(e_{2})&=ze_{1}.&  	
\end{align*}
\begin{center}
	\begin{tabular}{|c|c|c|c|c|}	
		\hline
		&	$ \varGamma_{\alpha^r\beta^l}(L_{1}^{11}) $&
\begin{minipage}{2cm} \vspace{0,3cm} Type of \\ $ \varGamma_{\alpha^0\beta^0}(L_{1}^{11}) $ \\ \end{minipage} &$ Der_{\alpha^r\beta^l}(L_{1}^{11}) $&CN\\
		\hline
		%%%%%%%%%%%%%%%%
	%%%%%%%%%%%%%%%%
$ l=0 $    &$ \begin{pmatrix}
c_1&c_2\\
0&c_1
\end{pmatrix} $	&Not small&$ \begin{pmatrix}
0&0\\
0&0
\end{pmatrix} $&Yes \\
\hline
	%%%%%%%%%%%%%%%%
$ l\geq1 $    &$ \begin{pmatrix}
0&c_2\\
0&0
\end{pmatrix} $	&&$ \begin{pmatrix}
0&d_2\\
0&0
\end{pmatrix} $& \\
\hline
\end{tabular}
\end{center}

%%%%%%%%%%%%%%%%%%%%%%Alg 3 cas 3 z1\neq0; t1=0	
\begin{align*}
L_{2}^{11}:[e_{1},e_{1}]&=0,&     [e_{1},e_{2}]&=0,
&
[e_2,e_{1}]&=e_{1}
,&       [e_2,e_{2}]&=0.&\\
\alpha_{34}(e_{1})&=e_{1},& \alpha(e_{2})&=e_{2}\, ,&
\beta(e_{1})&=0,& \beta(e_{2})&=e_{1}.&  	
\end{align*}
\begin{center}
	\begin{tabular}{|c|c|c|c|c|}	
		\hline
		&	$ \varGamma_{\alpha^r\beta^l}(L_{2}^{11}) $&
\begin{minipage}{2cm} \vspace{0,3cm} Type of \\ $ \varGamma_{\alpha^0\beta^0}(L_{2}^{11}) $ \\ \end{minipage} &$ Der_{\alpha^r\beta^l}(L_{2}^{11}) $&CN\\
		\hline
		%%%%%%%%%%%%%%%%
		%%%%%%%%%%%%%%%%
		$ l=0 $    &$ \begin{pmatrix}
		c_1&c_2\\
		0&c_1
		\end{pmatrix} $	&Not small&$ \begin{pmatrix}
		0&0\\
		0&0
		\end{pmatrix} $&Yes \\
		\hline
		%%%%%%%%%%%%%%%%
		$ l\geq1 $    &$ \begin{pmatrix}
		0&c_2\\
		0&0
		\end{pmatrix} $	&&$ \begin{pmatrix}
		0&d_2\\
		0&0
		\end{pmatrix} $& \\
		\hline
	\end{tabular}
\end{center}
%%%%%%%%%%%%%%%%%%%%%%Alg 3 cas 3 z1=0; t1\neq 0	
\begin{align*}
L_{3}^{11}:[e_{1},e_{1}]&=0,&     [e_{1},e_{2}]&=0,
&
[e_2,e_{1}]&=0
,&       [e_2,e_{2}]&=e_{1}.&\\
\alpha_{34}(e_{1})&=e_{1},& \alpha(e_{2})&=e_{2}\, ,&
\beta(e_{1})&=0,& \beta(e_{2})&=e_{1}.&  	
\end{align*}
\begin{center}
	\begin{tabular}{|c|c|c|c|c|}	
		\hline
		&	$ \varGamma_{\alpha^r\beta^l}(L_{3}^{11}) $&
\begin{minipage}{2cm} \vspace{0,3cm} Type of \\ $ \varGamma_{\alpha^0\beta^0}(L_{3}^{11}) $ \\ \end{minipage} &$ Der_{\alpha^r\beta^l}(L_{3}^{11}) $&CN\\
		\hline
		%%%%%%%%%%%%%%%%
		%%%%%%%%%%%%%%%%
		$ l=0 $    &$ \begin{pmatrix}
		c_1&c_2\\
		0&c_1
		\end{pmatrix} $	&Small&$ \begin{pmatrix}
		0&d_2\\
		0&0
		\end{pmatrix} $&Yes \\
		\hline
		%%%%%%%%%%%%%%%%
		$ l\geq1 $    &$ \begin{pmatrix}
		0&c_2\\
		0&0
		\end{pmatrix} $	&&$ \begin{pmatrix}
		0&d_2\\
		0&0
		\end{pmatrix} $& \\
		\hline
	\end{tabular}
\end{center}				
%%%%%%%%%%%%%%%%%%%%%%Alg 7 cas 3
\begin{align*}
L_{1}^{12}:[e_{1},e_{1}]&=0,&     [e_{1},e_{2}]&=e_{1},
&
[e_2,e_{1}]&=-e_{1}
,&       [e_2,e_{2}]&=-e_1.&\\
\alpha(e_{1})&=e_{1},& \alpha(e_{2})&=e_{2}\, ,&
\beta(e_{1})&=e_{1},& \beta(e_{2})&=e_{1}+e_{2}.&   	
\end{align*}
\begin{center}
	\begin{tabular}{|c|c|c|c|c|}	
		\hline
		&	$ \varGamma_{\alpha^r\beta^l}(L_{1}^{12}) $&
\begin{minipage}{2cm} \vspace{0,3cm} Type of \\ $ \varGamma_{\alpha^0\beta^0}(L_{1}^{12}) $ \\ \end{minipage} &$ Der_{\alpha^r\beta^l}(L_{1}^{12}) $&CN\\
		\hline
		%%%%%%%%%%%%%%%%
		%%%%%%%%%%%%%%%%
		$ l=0 $    &$ \begin{pmatrix}
		c_1&0\\
		0&c_1
		\end{pmatrix} $	&Small&$ \begin{pmatrix}
		0&d_2\\
		0&0
		\end{pmatrix} $&Yes \\
		\hline
		%%%%%%%%%%%%%%%%
		$ l\geq1 $    &$ \begin{pmatrix}
		c_1&lc_1\\
		0&c_1
		\end{pmatrix} $	&&$ \begin{pmatrix}
		0&d_2\\
		0&0
		\end{pmatrix} $& \\
		\hline
	\end{tabular}
\end{center}	

%%%%%%%%%%%%%%%%%%%%%%Alg 2 cas 4 y1\neq 0
\begin{align*}
L_{1}^{13}:[e_{1},e_{1}]&=0,&     [e_{1},e_{2}]&=e_{1},
&
[e_2,e_{1}]&=z_1e_{1}
,&       [e_2,e_{2}]&=t_1e_{1},&\\
\alpha(e_{1})&=0,& \alpha(e_{2})&=e_{1}\, ,&
\beta(e_{1})&=0,& \beta(e_{2})&=ze_{1}.&   	
\end{align*}
\begin{center}
	\begin{tabular}{|c|c|c|c|c|c|}	
		\hline
	&	&	$ \varGamma_{\alpha^r\beta^l}(L_{1}^{13}) $&
\begin{minipage}{2cm} \vspace{0,3cm} Type of \\ $ \varGamma_{\alpha^0\beta^0}(L_{1}^{13}) $ \\ \end{minipage} &$ Der_{\alpha^r\beta^l}(L_{1}^{13}) $&CN\\
	%	\hline
		%%%%%%%%%%%%%%%%
%$(r,l)\in \{(0,1),(1,0)\}  $&    &$ \begin{pmatrix}
%0&c_2\\
%0&0
%\end{pmatrix} $&	&$ \begin{pmatrix}
%0&d_2\\
%0&0
%\end{pmatrix} $& \\
\hline
		%%%%%%%%%%%%%%%%
$r=l=0  $ &$ z_1=-1 $   &$ \begin{pmatrix}
c_1&0\\
0&c_1
\end{pmatrix} $	&Small&$ \begin{pmatrix}
0&d_2\\
0&0
\end{pmatrix} $&Yes \\
\hline
		%%%%%%%%%%%%%%%%
$r=l=0  $ &$ z_1=0 $   &$ \begin{pmatrix}
c_1&0\\
0&c_1
\end{pmatrix} $	&Small&$ \begin{pmatrix}
0&0\\
0&0
\end{pmatrix} $&Yes \\
\hline
		%%%%%%%%%%%%%%%%
$r=l=0  $ &$ z_1\neq -1 $   &$ \begin{pmatrix}
c_1&0\\
0&c_1
\end{pmatrix} $	&Small&$ \begin{pmatrix}
0&0\\
0&0
\end{pmatrix} $ &Yes\\
\hline
		%%%%%%%%%%%%%%%%
		%%%%%%%%%%%%%%%%
\small{$(r,l)\in \{(0,1),(1,0)\}  $} &    &$ \begin{pmatrix}
0&c_2\\
0&0
\end{pmatrix} $&	&$ \begin{pmatrix}
0&d_2\\
0&0
\end{pmatrix} $& \\
\hline		
$r>1,\,l>1  $ &   &$ \begin{pmatrix}
0&c_2\\
0&0
\end{pmatrix} $&	&$ \begin{pmatrix}
0&d_2\\
0&0
\end{pmatrix} $& \\
\hline
	\end{tabular}
\end{center}
%%%%%%%%%%%%%%%%%%%%%%Alg 2 cas 4,y1=0 z1\neq 0
\begin{align*}
L_{2}^{13}:[e_{1},e_{1}]&=0,&     [e_{1},e_{2}]&=0,
&
[e_2,e_{1}]&=e_{1}
,&       [e_2,e_{2}]&=t_1e_{1},&\\
\alpha(e_{1})&=0,& \alpha(e_{2})&=e_{1}\, ,&
\beta(e_{1})&=0,& \beta(e_{2})&=ze_{1}.&   	
\end{align*}	
\begin{center}
	\begin{tabular}{|c|c|c|c|c|}	
		\hline
		&		$ \varGamma_{\alpha^r\beta^l}(L_{2}^{13}) $&
\begin{minipage}{2cm} \vspace{0,3cm} Type of \\ $ \varGamma_{\alpha^0\beta^0}(L_{2}^{13}) $ \\ \end{minipage} &$ Der_{\alpha^r\beta^l}(L_{2}^{13}) $&CN\\
		\hline
		%%%%%%%%%%%%%%%%
	%%%%%%%%%%%%%%%%
$r=l=0  $   &$ \begin{pmatrix}
c_1&c_2\\
0&c_1
\end{pmatrix} $	&Not small&$ \begin{pmatrix}
0&0\\
0&0
\end{pmatrix} $&Yes \\
\hline
%%%%%%%%%%%%%%%%
	%%%%%%%%%%%%%%%%
$(r,l)\in \{(0,1),(1,0)\}  $    &$ \begin{pmatrix}
0&c_2\\
0&0
\end{pmatrix} $	&&$ \begin{pmatrix}
0&d_2\\
0&0
\end{pmatrix} $& \\
\hline
		%%%%%%%%%%%%%%%%
$r>1,\,l>1  $  &$ \begin{pmatrix}
0&c_2\\
0&0
\end{pmatrix} $	&&$ \begin{pmatrix}
0&d_2\\
0&0
\end{pmatrix} $& \\
\hline
\end{tabular}
\end{center}	
%%%%%%%%%%%%%%%%%%%%%%Alg 2 cas 4 y1\neq 0
\begin{align*}
L_{3}^{13}:[e_{1},e_{1}]&=0,&     [e_{1},e_{2}]&=0,
&
[e_2,e_{1}]&=0
,&       [e_2,e_{2}]&=e_{1},&\\
\alpha(e_{1})&=0,& \alpha(e_{2})&=e_{1}\, ,&
\beta(e_{1})&=0,& \beta(e_{2})&=ze_{1}.&   	
\end{align*}
\begin{center}
	\begin{tabular}{|c|c|c|c|c|}	
		\hline
		&		$ \varGamma_{\alpha^r\beta^l}(L_{3}^{13}) $&
\begin{minipage}{2cm} \vspace{0,3cm} Type of \\ $ \varGamma_{\alpha^0\beta^0}(L_{3}^{13}) $ \\ \end{minipage} &$ Der_{\alpha^r\beta^l}(L_{3}^{13}) $&CN\\
		\hline
		%%%%%%%%%%%%%%%%
		%%%%%%%%%%%%%%%%
		$r=l=0  $   &$ \begin{pmatrix}
		c_1&c_2\\
		0&c_1
		\end{pmatrix} $	&Small&$ \begin{pmatrix}
		0&d_2\\
		0&0
		\end{pmatrix} $&Yes \\
		\hline
		%%%%%%%%%%%%%%%%
		%%%%%%%%%%%%%%%%
		$(r,l)\neq (0,0)  $    &$ \begin{pmatrix}
		0&c_2\\
		0&0
		\end{pmatrix} $	&&$ \begin{pmatrix}
		0&d_2\\
		0&0
		\end{pmatrix} $ &\\
		\hline
		%%%%%%%%%%%%%%%%
	%	$r>1,\,l>1  $  &$ \begin{pmatrix}
	%	0&c_2\\
	%	0&0
	%	\end{pmatrix} $	&&$ \begin{pmatrix}
	%	0&d_2\\
	%	0&0
	%	\end{pmatrix} $& \\
		%\hline
	\end{tabular}
\end{center}			
%==================================			
%%%%%%%%%%%%%%%%%%%%%%Alg 2 cas 4
%	\begin{align*}
%	L_{24}:[e_{1},e_{1}]_{36}&=0&,     [e_{1},e_{2}]_{36}&=0,
%	&
%	[e_2,e_{1}]_{36}&=e_{1}
%	,&       [e_2,e_{2}]_{36}&=t_1e_{1},&\\
%	\alpha_{36}(e_{1})&=e_{1},& \alpha_{36}(e_{2})&=e_{1}\, ,&
%	\beta_{36}(e_{1})&=0,& \beta_{36}(e_{2})&=ze_{1}.&   	
%	\end{align*}
%%%%%%%%%%%%%%%%%%%%%%Alg 12 cas 4
\begin{align*}
L_{1}^{14}:[e_{1},e_{1}]&=0,&     [e_{1},e_{2}]&=0,
&
[e_2,e_{1}]&=0
,&       [e_2,e_{2}]&=e_{1},&\\
\alpha(e_{1})&=e_{1},& \alpha(e_{2})&=e_{1}+e_{2}\, ,&
\beta(e_{1})&=0,& \beta(e_{2})&=ze_{1}.&   	
\end{align*}
\begin{center}
	\begin{tabular}{|c|c|c|c|c|}	
		\hline
		&		$ \varGamma_{\alpha^r\beta^l}(L_{1}^{14}) $&
\begin{minipage}{2cm} \vspace{0,3cm} Type of \\ $ \varGamma_{\alpha^0\beta^0}(L_{1}^{14}) $ \\ \end{minipage} &$ Der_{\alpha^r\beta^l}(L_{1}^{14}) $&CN\\
		\hline
		%%%%%%%%%%%%%%%%
		%%%%%%%%%%%%%%%%
		$r=l=0  $   &$ \begin{pmatrix}
		c_1&c_2\\
		0&c_1
		\end{pmatrix} $	&Not small&$ \begin{pmatrix}
		0&d_2\\
		0&0
		\end{pmatrix} $&Yes \\
		\hline
		%%%%%%%%%%%%%%%%
		%%%%%%%%%%%%%%%%
		$l\geq 1 $    &$ \begin{pmatrix}
		0&c_2\\
		0&0
		\end{pmatrix} $	&&$ \begin{pmatrix}
		0&d_2\\
		0&0
		\end{pmatrix} $ &\\
		\hline
	\end{tabular}
\end{center}
%%%%%%%%%%%%%%%%%%%%%%Alg 17 cas 4,y1\neq 0	
\begin{align*}
L_{1}^{15}:[e_{1},e_{1}]&=0,&     [e_{1},e_{2}]&=e_{1},
&
[e_2,e_{1}]&=0
,&       [e_2,e_{2}]&=t_{1}e_{1},&\\
\alpha(e_{1})&=0,& \alpha(e_{2})&=e_{1}\, ,&
\beta(e_{1})&=e_{1},& \beta(e_{2})&=e_{2}.&    	
\end{align*}
\begin{center}
	\begin{tabular}{|c|c|c|c|c|}	
		\hline
		&		$ \varGamma_{\alpha^r\beta^l}(L_{1}^{15}) $&
\begin{minipage}{2cm} \vspace{0,3cm} Type of \\ $ \varGamma_{\alpha^0\beta^0}(L_{1}^{15}) $ \\ \end{minipage} &$ Der_{\alpha^r\beta^l}(L_{1}^{15}) $&CN\\
		\hline
		%%%%%%%%%%%%%%%%
		%%%%%%%%%%%%%%%%
		$r=0 ,\,l\geq 0 $   &$ \begin{pmatrix}
		c_1&0\\
		0&c_1
		\end{pmatrix} $	&Small&$ \begin{pmatrix}
		0&0\\
		0&0
		\end{pmatrix} $&Yes \\
		\hline
		%%%%%%%%%%%%%%%%
		%%%%%%%%%%%%%%%%
		$r\geq 1 $    &$ \begin{pmatrix}
		0&c_2\\
		0&0
		\end{pmatrix} $	&&$ \begin{pmatrix}
		0&d_2\\
		0&0
		\end{pmatrix} $ &\\
		\hline
	\end{tabular}
\end{center}
%%%%%%%%%%%%%%%%%%%%%%Alg 17 cas 4 y1=0	
%\begin{align*}
%	L_{2}^{22}:[e_{1},e_{1}]&=0&,     [e_{1},e_{2}]_{37}&=0,
%	&
%	[e_2,e_{1}]&=0
%	,&       [e_2,e_{2}]_{37}&=e_{1},&\\
%	\alpha(e_{1})&=0,& \alpha(e_{2})&=e_{1}\, ,&
%	\beta(e_{1})&=e_{2},& \beta(e_{2})&=e_{2}.&    	
%	\end{align*}
%	%%%%%%%%%%%%%%%%%%%%%%Alg 17 cas 4	
%	\begin{align*}
%	L_{28}:[e_{1},e_{1}]_{38}&=0&,     [e_{1},e_{2}]_{38}&=e_{1},
%	&
%	[e_2,e_{1}]_{38}&=0
%	,&       [e_2,e_{2}]_{38}&=0,&\\
%	\alpha_{38}(e_{1})&=0,& \alpha_{38}(e_{2})&=e_{1}\, ,&
%	\beta_{38}(e_{1})&=e_{1},& \beta_{38}(e_{2})&=e_{2}.&    	
%	\end{align*}
%	\begin{align*}
%	L_{29}:[e_{1},e_{1}]_{39}&=0&,     [e_{1},e_{2}]_{39}&=0,
%	&
%	[e_2,e_{1}]_{39}&=0
%	,&       [e_2,e_{2}]_{39}&=e_{1},&\\
%	\alpha_{39}(e_{1})&=0,& \alpha_{39}(e_{2})&=e_{1}\, ,&
%	\beta_{39}(e_{1})&=e_{1},& \beta_{39}(e_{2})&=e_{2}.&    	
%	\end{align*}
%%%%%%%%%%%%%%%%%%%%%%Alg 18 cas 4	
\begin{align*}
L_{1}^{16}:[e_{1},e_{1}]&=0,&     [e_{1},e_{2}]&=0,
&
[e_2,e_{1}]&=0
,&       [e_2,e_{2}]&=e_{1},& \\
\alpha(e_{1})&=0,& \alpha(e_{2})&=e_{1}\, ,&
\beta(e_{1})&=e_{1},& \beta(e_{2})&=ze_{1}+e_{2}.&   	
\end{align*}
\begin{center}
	\begin{tabular}{|c|c|c|c|c|}	
		\hline
		&		$ \varGamma_{\alpha^r\beta^l}(L_{1}^{16}) $&
\begin{minipage}{2cm} \vspace{0,3cm} Type of \\ $ \varGamma_{\alpha^0\beta^0}(L_{1}^{16}) $ \\ \end{minipage} &$ Der_{\alpha^r\beta^l}(L_{1}^{16}) $&CN\\
		\hline
		%%%%%%%%%%%%%%%%
		%%%%%%%%%%%%%%%%
		$r=0,\,l\geq 0 $   &$ \begin{pmatrix}
		c_1&c_2\\
		0&c_1
		\end{pmatrix} $	&Small&$ \begin{pmatrix}
		0&c_2\\
		0&0
		\end{pmatrix} $&Yes \\
		\hline
		%%%%%%%%%%%%%%%%
		%%%%%%%%%%%%%%%%
		$r\geq 1 $    &$ \begin{pmatrix}
		0&c_2\\
		0&0
		\end{pmatrix} $	&&$ \begin{pmatrix}
		0&d_2\\
		0&0
		\end{pmatrix} $& \\
		\hline
	\end{tabular}
\end{center}
%%%%%%%%%%%%%%%%%%%%%%Alg 19 cas 4	
\begin{align*}
L_{1}^{17}:[e_{1},e_{1}]&=0, & [e_{1},e_{2}]&=e_{1}, & [e_2,e_{1}]&=-e_{1}, & [e_2,e_{2}]&=(1-z)e_{1},\\
\alpha(e_{1})&=e_{1},& \alpha(e_{2})&=e_{1}+e_{2},& \beta(e_{1})&=e_{1},& \beta(e_{2})&=ze_{1}+e_{2}.  	
\end{align*}
\begin{center}
	\begin{tabular}{|c|c|c|c|}	
		\hline
				$ \varGamma_{\alpha^r\beta^l}(L_{1}^{17}) $&
\begin{minipage}{2cm} \vspace{0,3cm} Type of \\ $ \varGamma_{\alpha^0\beta^0}(L_{1}^{17}) $ \\ \end{minipage} &$ Der_{\alpha^r\beta^l}(L_{1}^{17}) $&CN\\
		\hline
		%%%%%%%%%%%%%%%%
		%%%%%%%%%%%%%%%%
	   $ \begin{pmatrix}
		c_1&(lz+r)c_1\\
		0&c_1
		\end{pmatrix} $	&Small&$ \begin{pmatrix}
		0&c_2\\
		0&0
		\end{pmatrix} $&Yes \\
		\hline
	\end{tabular}
\end{center}	
\end{theorem}
\begin{cor} The following statements hold.
\begin{enumerate}[label=\upshape{(\roman*)},left=0pt]
\item  The dimensions of the centroids of $2$-dimensional BiHom-Lie Algebras vary between one  and two.
\item  Every $ 2 $-dimensional multiplicative BiHom-Lie algebra have a small centroid if and only if it isomorphic
to one of the following  BiHom-Lie algebras:
\begin{align*}
L_{1}^{1}(z_1)(z_1\neq 0),\,L_{2}^{1},\,L_{4}^{1}(z_1)(z_1 \neq 0),\,L_{5}^{1}, \\ \,L_{1}^{8},\,L_{1}^{10},\,L_{3}^{11},\,L_{1}^{12},\,L_{1}^{13},\,L_{1}^{15},\,L_{1}^{16},\,L_{1}^{17}.
\end{align*}
\item The dimensions of the derivations of $2$-dimensional BiHom-Lie algebras vary between zero  and two.
\item Every $ 2 $-dimensional multiplicative BiHom-Lie algebra is characteristically nilpotent
if and only if it is not isomorphic to $ L_1^{10} $. 	 	
\end{enumerate}
\end{cor}

%%%%%%%%%%%%%%%%%%%%%%%%%%


\begin{thebibliography}{999}	
\bibitem{AbdaouiAmmarMakhloufCohhomLiecolalg2015}
Abdaoui, K., Ammar, F., Makhlouf, A.: Constructions and cohomology of Hom-Lie color algebras, Comm. Algebra, \textbf{43}, 4581-4612 (2015)
\bibitem{belhsine}
Abdelkader, B. H.: Generalized derivations of BiHom-Lie algebras,
J. Gen. Lie Theory Appl. \textbf{11}(1), 1-7 (2017)
\bibitem{AbdulkarimAssociative}
Abdulkareem, A. O., Fiidow, M. A., Rakhimov, I. S.: Derivations and centroids of four-dimensional associative algebras,
Intern. J. Pure and Appl. Math. \textbf{112}(4), 655-671 (2017)
\bibitem{AbramovJNMPh2006:gradedqdifalg}
Abramov, V.: On a graded $q$-differential algebra, J. Nonlinear Math. Phys. \textbf{13}(sup 1), 1-8 (2006) % DOI: 10.2991/jnmp.2006.13.s.1
\bibitem{AbramovGLTMPBSpr2009:GradedqDiffAlgqConnect}
Abramov, V.: Graded $q$-differential algebra approach to $q$-connection, In: Silvestrov, S., Paal, E., Abramov, V., Stolin, A. (Eds.),
Generalized Lie Theory in Mathematics, Physics and Beyond, Springer-Verlag, Berlin, Heidelberg, Ch. 6, 71-79 (2009)
\bibitem{AbramovRaknuz2016EngMathSpr:SemicomGaloisExtRedQPl}
Abramov, V., Raknuzzaman, Md.: Semi-commutative Galois extensions and reduced quantum plane, In: Silvestrov S., Rancic M. (Eds.), Engineering Mathematics II, Springer Proceedings in Mathematics and Statistics, \textbf{179}, Springer, Cham, 13-31 (2016)
\bibitem{AbramovSilvestrov:3homLiealgsigmaderivINvol}
Abramov, V., Silvestrov, S.: $3$-Hom-Lie algebras based on $\sigma$-derivation and involution, Adv. Appl. Clifford Algebras, \textbf{30}(45) (2020)
\bibitem{AizawaSaito}
Aizawa, N., Sato, H.: $q$-Deformation of the Virasoro algebra with central extension, Phys. Lett. B, \textbf{256}, 185-190 (1991) (Hiroshima University preprint, preprint HUPD-9012 (1990))
\bibitem{AlmutariAhmad:CentrquasicentrfdimLeibnizalg}
Almutari, H., Ahmad, A. G.:
Centroids and quasi-centroids of finite-dimensional Leibniz algebras, Intern. J. Pure and Appl. Math. \textbf{113}(2), 203-217 (2017)
\bibitem{AmmarEjbehiMakhlouf:homdeformation}
Ammar, F., Ejbehi, Z., Makhlouf, A.: Cohomology and deformations of Hom-algebras, J. Lie Theory, \textbf{21}(4), 813-836 (2011)
\bibitem{AmmarMabroukMakhloufCohomnaryHNLalg2011}
Ammar, F., Mabrouk, S., Makhlouf, A.: Representations and cohomology of $n$-ary multiplicative Hom-Nambu-Lie algebras, J. Geom. Phys. \textbf{61}(10), 1898-1913 (2011) %DOI: 10.1016/j.geomphys.2011.04.022
\bibitem{AmmarMakhloufHomLieSupAlg2010}
Ammar, F., Makhlouf, A.: Hom-Lie superalgebras and Hom-Lie admissible superalgebras, J. Algebra, \textbf{324}(7), 1513-1528 (2010)
\bibitem{AmmarMakhloufSaadaoui2013:CohlgHomLiesupqdefWittSup}
Ammar, F., Makhlouf A., Saadaoui, N.: Cohomology of Hom-Lie superalgebras and $q$-deformed Witt superalgebra, Czechoslovak Math. J. \textbf{68}, 721-761 (2013)
\bibitem{AmmarMakhloufSilv:TernaryqVirasoroHomNambuLie}
Ammar, F., Makhlouf, A., Silvestrov, S.: Ternary $q$-Virasoro-Witt Hom-Nambu-Lie algebras, J. Phys. A: Math. Theor. \textbf{43}(26), 265204 (2010)
%https://doi.org/10.1088/1751-8113/43/26/265204
\bibitem{ArmakanFarhangdoost:IJGMMP}
Armakan A., Farhangdoost, M. R.: Geometric aspects of extensions of Hom-Lie superalgebras, Int. J. Geom. Methods Mod. Phys. \textbf{14}, 1750085 (2017)
\bibitem{ArmakanSilv:NondegKillingformsHomLiesuperalg}
Armakan A., Farhangdoost, M. R., Silvestrov S.: Non-degenerate Killing forms on Hom-Lie superalgebras, arXiv:2010.01778v2 [math.RA] (2020)
\bibitem{ArmakanSilv:envelalgcertaintypescolorHomLie}
Armakan A., Silvestrov S.: Enveloping algebras of certain types of color Hom-Lie algebras. In: Silvestrov, S., Malyarenko, A., Ran\u{c}i\'{c}, M. (Eds.), Algebraic Structures and Applications, Springer Proceedings in Mathematics and Statistics, \textbf{317}, Springer, Ch. 10, 257-284 (2020)
\bibitem{ArmakanSilvFarh:envelopalgcolhomLiealg}
Armakan, A., Silvestrov, S., Farhangdoost, M. R.: Enveloping algebras of color Hom-Lie algebras, Turk. J. Math. \textbf{43}, 316-339 (2019)
(arXiv:1709.06164 [math.QA] (2017)) %doi:10.3906/mat-1808-96.
\bibitem{ArmakanSilvFarh:exthomLiecoloralg}
Armakan, A., Silvestrov, S., Farhangdoost, M. R.: Extensions of Hom-Lie color algebras, Georgian Math. J. (2019),  doi:10.1515/gmj-2019-2033, (arXiv:1709.08620 [math.QA] (2017))
\bibitem{akms:ternary}
Arnlind, J., Kitouni, A., Makhlouf, A., Silvestrov, S.:
Structure and cohomology of $3$-Lie algebras induced by Lie algebras, In: Makhlouf, A., Paal, E., Silvestrov, S. D., Stolin, A., Algebra, Geometry and Mathematical Physics, Springer Proceedings in Mathematics and Statistics, \textbf{85}, Springer, 123-144 (2014)
\bibitem{ams:ternary}
Arnlind, J., Makhlouf, A., Silvestrov, S.:
Ternary Hom-Nambu-Lie algebras induced by Hom-Lie algebras, J. Math. Phys. \textbf{51}(4), 043515 (2010)
\bibitem{ArnlindMakhloufSilvnaryHomLieNambuJMP2011}
Arnlind, J., Makhlouf, A. Silvestrov, S.: Construction of $n$-Lie algebras and $n$-ary Hom-Nambu-Lie algebras, J. Math. Phys. \textbf{52}(12), 123502 (2011)
\bibitem{Bakayoko2014:ModulescolorHomPoisson}
Bakayoko, I.: Modules over color Hom-Poisson  algebras, J. Gen. Lie  Theory Appl. \textbf{8}(1), 1000212 (2014) %doi:10.4172/1736-4337.1000212
\bibitem{Bakayoko:LaplacehomLiequasibialg}
Bakayoko, I.: Laplacian of Hom-Lie quasi-bialgebras, Intern. J. Algebra, \textbf{8} (15), 713-727 (2014)
\bibitem{Bakayoko:LmodcomodhomLiequasibialg}
Bakayoko, I.: $L$-modules, $L$-comodules and Hom-Lie quasi-bialgebras, African Diaspora J. Math. \textbf{17}, 49-64 (2014)
\bibitem{BakayokoDialo2015:genHomalgebrastr} Bakayoko, I., Diallo, O. W.: Some generalized Hom-algebra structures, J. Gen. Lie Theory Appl.
\textbf{9}(1), 1000226 (2015)
%doi: 10.4172/1736-4337. 1000226
\bibitem{BakyokoSilvestrov:Homleftsymmetriccolordialgebras}
Bakayoko, I., Silvestrov, S.: Hom-left-symmetric color dialgebras, Hom-triden\-driform color algebras and Yau's twisting generalizations,
arXiv:1912.01441 [math.RA] (2019)
\bibitem{BakyokoSilvestrov:MultiplicnHomLiecoloralg}
Bakayoko, I., Silvestrov, S.: Multiplicative $n$-Hom-Lie color algebras,
In: Silvestrov, S., Malyarenko, A., Ran\u{c}i\'{c}, M. (Eds.), Algebraic Structures and Applications, Springer Proceedings in Mathematics and Statistics, \textbf{317}, Springer, Ch. 7, 159-187 (2020). (arXiv:1912.10216[math.QA])
\bibitem{BakayokoToure2019:genHomalgebrastr} Bakayoko, I., Tour\'e, B. M.: Constructing Hom-Poisson color algebras, Int. J. Algebra, \textbf{13}(1), 1-16 (2019)
\bibitem{BenAbdeljElhamdKaygorMakhl201920GenDernBiHomLiealg} Ben Abdeljelil, A., Elhamdadi, M., Kaygorodov, I., Makhlouf, A.,
Generalized derivations of $n$-BiHom-Lie algebras, In: Silvestrov, S., Malyarenko, A., Rancic, M. (Eds.), Algebraic Structures and Applications,  Springer Proceedings in Mathematics and Statistics, Vol 317, Ch. 4, 2020. arXiv:1901.09750[math.RA]
\bibitem{BeitesKaygorodovPopov} Beites, P. D., Kaygorodov, I., Popov, Y.:
Generalized derivations of multiplicative $n$-ary Hom-$\Omega$ color algebras, Bull. of the Malay. Math. Sci. Soc. \textbf{41} (2018), %DOI: 10.1007/s40840-017-0486-8
\bibitem{BenHassineChtiouiMabroukNcib:Strcohom3LieRinehartsuperalg}
Ben Hassine, A., Chtioui, T., Mabrouk S., Silvestrov, S.: Structure and cohomology of $3$-Lie-Rinehart superalgebras, arXiv:2010.01237 [math.RA] (2020)
\bibitem{BenHassineMabroukNcib:ConstrMultiplicnaryhomNambualg}
Ben Hassine, A., Mabrouk S., Ncib, O.: Some Constructions of Multiplicative $n$-ary
    hom-Nambu Algebras, Adv. Appl. Clifford Algebras, \textbf{29}(88) (2019) % DOI:10.1007/s00006-019-0996-6
\bibitem{BenMakh:Hombiliform}
Benayadi, S., Makhlouf, A.: Hom-Lie algebras with symmetric invariant nondegenerate bilinear forms, J. Geom. Phys. \textbf{76}, 38-60 (2014)
\bibitem{CanepeelGoyaverts:MonoidalHomHopfalgebras}	
Caenepeel, S., Goyvaerts, I.: Monoidal Hom-Hopf algebras, Comm. Algebra, \textbf{39}(6), 2216–2240 (2011)
\bibitem{CaoChen2012:SplitregularhomLiecoloralg}
Cao, Y., Chen, L.: On split regular Hom-Lie color algebras, Comm. Algebra \textbf{40}, 575-592 (2012)
\bibitem{ChaiElinPop}
Chaichian, M., Ellinas, D., Popowicz, Z.: Quantum conformal algebra with central extension, Phys. Lett. B, \textbf{248}, 95-99 (1990)
\bibitem{ChaiIsLukPopPresn}
Chaichian, M., Isaev, A. P., Lukierski, J., Popowic, Z., Pre\v{s}najder, P.: $q$-Defor\-mations of Virasoro algebra and conformal dimensions, Phys. Lett. B,  \textbf{262} (1), 32-38 (1991)
\bibitem{ChaiKuLuk}
Chaichian, M., Kulish, P., Lukierski, J.: $q$-Deformed Jacobi identity, $q$-oscillators and $q$-deformed infinite-dimensional algebras, Phys. Lett. B,  \textbf{237}, 401-406 (1990)
\bibitem{ChaiPopPres}
Chaichian, M., Popowicz, Z., Pre\v{s}najder, P.: $q$-Virasoro algebra and its relation to the $q$-deformed KdV system, Phys. Lett. B, \textbf{249}, 63-65 (1990)
\bibitem{ChenMaNi:GenDerLieColAlg}
Chen, L., Ma, Y., Ni, L.: Generalized Derivations of Lie color algebras, Results Math. \textbf{63} (3-4), 923-936 (2013)
\bibitem{Yongsheng}
Cheng, Y., Qi, H.: Representations of Bihom-Lie algebras, arXiv:1610.04302 (2016)
\bibitem{CurtrZachos1}
Curtright, T. L., Zachos, C. K.: Deforming maps for quantum algebras, Phys. Lett. B, \textbf{243}, 237-244 (1990)
\bibitem{DamKu}
Damaskinsky, E. V., Kulish, P. P.: Deformed oscillators and their applications, Zap. Nauch. Semin. LOMI, \textbf{189}, 37-74 (1991) (in Russian) [Engl. tr. in J. Sov. Math., \textbf{62}, 2963-2986 (1992)
\bibitem{DalTakh} Daletskii, Y. L., Takhtajan, L. A.: Leibniz and Lie algebra structures for Nambu algebra, Lett. Math. Phys. \textbf{39}, 127-141 (1997)
\bibitem{DaskaloyannisGendefVir}
Daskaloyannis, C.: Generalized deformed Virasoro algebras, Modern Phys. Lett. A, \textbf{7}(9), 809-816 (1992)
\bibitem{DoradodoGD}
Dorado-Aguilar, E., Garcia-Delgado, R., Martinez-Sigala, E., Rodríguez-Vallarte, M. C., Salgado, G.:
Generalized derivations and some structure theorems for Lie algebras, J. Algebra Appl. \textbf{19}(2), 2050024, 18 pp (2020)
\bibitem{ElchingerLundMakhSilv:BracktausigmaderivWittVir}
Elchinger, O., Lundeng{\aa}rd, K., Makhlouf, A., Silvestrov, S. D.:
Brackets with $(\tau,\sigma)$-derivations and $(p,q)$-deformations of Witt and Virasoro algebras, Forum Math. \textbf{28}(4), 657-673 (2016)
\bibitem{Filippov:nLie}
Filippov, V. T.: $n$-Lie algebras, Siberian Math. J. \textbf{26}, 879-891 (1985)
(Transl. from Russian: Sibirskii Matematicheskii Zhurnal, \textbf{26}, 126-140 (1985))
\bibitem{Filippov1998:deltaderivLiealg}
Filippov, V. T.: On $\delta$-derivations of Lie algebras. Sib. Math. J. \textbf{39}, 1218-1230 (1998) %DOI:10.1007/BF02674132
(Transl. from Russian: Sibirskii Matematicheskii Zhurnal, \textbf{39}(6), 1409-1422 (1998))
\bibitem{Filippov1999:deltaderivprimeLiealg}
Filippov, V. T.: $\delta$-Derivations of prime Lie algebras, Siberian Math. J. \textbf{40}, 174-184 (1999)
\bibitem{Filippov2000:deltaderivprimealternMalcevalg}
Filippov, V. T.: $\delta$-derivations of prime alternative and Mal'tsev algebras, Algebra and Logic \textbf{39}, 354-358 (2000). % DOI:10.1007/BF02681620
(Transl. from Russian: Algebra i Logika, \textbf{39}(5), 618-625 (2000))
\bibitem{GeorgeEugen:GenderivLiealg}
George, F. L.,  Eugene, M. L.: Generalized derivations of Lie algebras, J. Algebra \textbf{228}(1) , 165-203 (2000)
\bibitem{GrazMakhlMeniniPanaite:bihom}
Graziani, G.,  Makhlouf, A., Menini, C., Panaite, F.: BiHom-associative algebras, BiHom-Lie algebras and BiHom-bialgebras,
SIGMA (Symmetry, Integrability and Geometry: Methods and Applications), \textbf{11}(086), 34 pp (2015)
\bibitem{GuanChenSun:HomLieSuperalgebras}
Guan, B., Chen, L., Sun, B.: On Hom-Lie superalgebras, Adv. Appl. Clifford Algebras, \textbf{29}(16) (2019)
\bibitem{HartwigLarssonSilvestrov:defLiealgsderiv}
Hartwig, J. T., Larsson, D., Silvestrov, S. D.:
Deformations of Lie algebras using $\sigma$-derivations, J. Algebra, \textbf{295}(2),  314-361 (2006) (Preprint in Mathematical Sciences 2003:32, LUTFMA-5036-2003, Centre for Mathematical Sciences, Department of Mathematics, Lund Institute of Technology, 52 pp. (2003))
\bibitem{HounkonnouHoundedjiSilvestrov:DoubleconstrbiHomFrobalg}
Hounkonnou,  M. N., Houndedji, G. D., Silvestrov, S.: Double constructions of biHom-Frobenius algebras, arXiv:2008.06645 [math.QA] (2020)
\bibitem{Hu}
Hu, N.: $q$-Witt algebras, $q$-Lie algebras, $q$-holomorph structure and representations,  Algebra Colloq. \textbf{6}(1), 51-70 (1999)
\bibitem{Kassel92}
Kassel, C.: Cyclic homology of differential operators, the Virasoro algebra and a $q$-analogue, Comm. Math. Phys. \textbf{146}(2), 343-356 (1992)
\bibitem{Kaygorodov2012:deltaDerivnaryalg}
Kaygorodov, I.: On $\delta$-Derivations of $n$-ary algebras, Izvestiya: Mathematics, \textbf{76}(6), 1150-1162 (2012)
\bibitem{Kaygorodov2011:nplus1Aryderivsimplenaryalg}
Kaygorodov, I.: $(n + 1)$-Ary derivations of simple $n$-ary algebras, Algebra and Logic, \textbf{50}(5), 470-471 (2011)
\bibitem{Kaygorodov2014:nplus1AryderivsemisimpleFilipovalg}
Kaygorodov, I.: $(n + 1)$-Ary derivations of semisimple Filippov algebras, Math. Notes, \textbf{96}(2), 208-216 (2014)
\bibitem{KaygorodovPopov2014:Altalgadmderiv}
Kaygorodov, I, Popov, Yu.: Alternative algebras admitting derivations with invertible values and invertible derivations, Izv. Math. \textbf{78}, 922-935 (2014)
\bibitem{KaygorodovPopov2016:GeneralizedderivcolornLiealg}
Kaygorodov, I. Popov, Yu.: Generalized derivations of (color) $n$-ary algebras, Linear and Multilinear Algebra, \textbf{64}(6), 1086-1106 (2016)
\bibitem{KitouniMakhloufSilvestrov}
Kitouni, A., Makhlouf, A., Silvestrov, S.: On $(n+1)$-Hom-Lie algebras induced by $n$-Hom-Lie algebras, Georgian Math. J. \textbf{23}(1), 75-95 (2016)
\bibitem{kms:narygenBiHomLieBiHomassalgebras2020}
Kitouni, A., Makhlouf, A., Silvestrov, S.: On $n$-ary generalization of BiHom-Lie algebras and BiHom-associative algebras, In: Silvestrov, S., Malyarenko, A., Rancic, M. (Eds.), Algebraic Structures and Applications, Springer Proceedings in Mathematics and Statistics, \textbf{317}, Ch 5 (2020)
\bibitem{kms:solvnilpnhomlie2020}
Kitouni, A., Makhlouf, A., Silvestrov, S.: On solvability and nilpotency for $n$-Hom-Lie algebras and $(n+1)$-Hom-Lie algebras induced by $n$-Hom-Lie algebras, In: Silvestrov, S., Malyarenko, A., Rancic, M. (Eds.), Algebraic Structures and Applications,
Springer Proceedings in Mathematics and Statistics, \textbf{317}, Springer, Ch  6, 127-157 (2020)
\bibitem{LarssonSigSilvJGLTA2008:QuasiLiedefFttN}
Larsson, D., Sigurdsson, G., Silvestrov, S. D.: Quasi-Lie deformations on the algebra $\mathbb{F}[t]/(t^N)$, J. Gen. Lie Theory Appl. \textbf{2}(3), 201-205 (2008)
\bibitem{LarssonSilvJA2005:QuasiHomLieCentExt2cocyid}
Larsson, D., Silvestrov, S. D.: Quasi-Hom-Lie algebras, central extensions and $2$-cocycle-like identities, J. Algebra \textbf{288}, 321-344 (2005) (Preprints in Mathematical Sciences 2004:3, LUTFMA-5038-2004, Centre for Mathematical Sciences, Department of Mathematics, Lund Institute of Technology, Lund University (2004))
\bibitem{LarssonSilv2005:QuasiLieAlg}
Larsson, D., Silvestrov, S. D.: Quasi-Lie algebras, In: Fuchs, J., Mickelsson, J., Rozenblioum, G., Stolin, A., Westerberg, A. (Eds.),
Noncommutative Geometry and Representation Theory in Mathematical Physics, Contemp. Math. \textbf{391}, Amer. Math. Soc., Providence, RI, 241-248 (2005) (Preprints in Mathematical Sciences 2004:30, LUTFMA-5049-2004, Centre for Mathematical Sciences, Department of Mathematics, Lund Institute of Technology, Lund University (2004))
\bibitem{LarssonSilv:GradedquasiLiealg}
Larsson, D., Silvestrov, S. D.: Graded quasi-Lie agebras, Czechoslovak J. Phys. \textbf{55}, 1473-1478 (2005)
\bibitem{LarssonSilv:QuasidefSl2}
Larsson, D., Silvestrov, S. D.: Quasi-deformations of $sl_2(\mathbb{F})$ using twisted derivations, Comm. Algebra, \textbf{35}, 4303-4318 (2007)
\bibitem{LarssonSilvestrovGLTMPBSpr2009:GenNComplTwistDer}
Larsson, D., Silvestrov, S. D.: On generalized $N$-complexes comming from twisted derivations, In: Silvestrov, S., Paal, E., Abramov, V., Stolin, A. (Eds.),
Generalized Lie Theory in Mathematics, Physics and Beyond, Springer-Verlag, Ch. 7, 81-88 (2009)
\bibitem{LegerLuks:GenDerivLiealg}
Leger, G., Luks, E., Generalized derivations of Lie algebras, J. Algebra \textbf{228}, 165-203 (2000)
\bibitem{LiuKQuantumCentExt}
Liu, K. Q.: Quantum central extensions, C. R. Math. Rep. Acad. Sci. Canada \textbf{13}(4), 135-140 (1991)
\bibitem{LiuKQCharQuantWittAlg}
Liu, K. Q.: Characterizations of the quantum Witt algebra, Lett. Math. Phys. \textbf{24}(4), 257-265 (1992)
\bibitem{LiuKQPhDthesis}
Liu, K. Q.: The quantum Witt algebra and quantization of some modules over Witt algebra, PhD Thesis, Department of Mathematics, University of Alberta, Edmonton, Canada (1992)
\bibitem{MaMakhSil:CurvedOoperatorSyst} Ma, T., Makhlouf, A., Silvestrov, S.:
Curved $\mathcal{O}$-operator systems, arXiv: 1710.05232, (2017)
\bibitem{MaMakhSil:RotaBaxbisyscovbialg} Ma, T., Makhlouf, A., Silvestrov, S.:
Rota-Baxter bisystems and covariant bialgebras, arXiv:1710.05161[math.RA] (2017)
\bibitem{MaMakhSil:RotaBaxCosyCoquasitriMixBial} Ma, T., Makhlouf, A., Silvestrov, S.:
Rota-Baxter cosystems and coquasitriangular mixed bialgebras, J. Algebra Appl. doi:10.1142/S021949882150064X (Accepted 2019)	
\bibitem{MabroukNcibSilvestrov2020:GenDerRotaBaxterOpsnaryHomNambuSuperalgs}
Mabrouk, S., Ncib, O., Silvestrov, S.: Generalized derivations and Rota-Baxter operators of $n$-ary Hom-Nambu superalgebras, arXiv:2003.01080[math.QA]
\bibitem{Makhlouf2010:ParadigmnonassHomalgHomsuper}
Makhlouf, A.: Paradigm of nonassociative Hom-algebras and Hom-superalgebras,
Proceedings of Jordan Structures in Algebra and Analysis Meeting, 145-177 (2010)
\bibitem{ms:homstructure}
Makhlouf, A., Silvestrov, S. D.: Hom-algebra structures, J. Gen. Lie Theory Appl. \textbf{2}(2), 51-64 (2008)
(Preprints in Mathematical Sciences  2006:10, LUTFMA-5074-2006, Centre for Mathematical Sciences, Department of Mathematics, Lund Institute of Technology, Lund University (2006))
\bibitem{MakhSil:HomHopf}
Makhlouf, A., Silvestrov, S.:
Hom-Lie admissible Hom-coalgebras and Hom-Hopf algebras, In: Silvestrov, S., Paal, E., Abramov, V., Stolin, A. (Eds.),
Generalized Lie Theory in Mathematics, Physics and Beyond, Springer-Verlag, Berlin, Heidelberg, Ch. 17, 189-206 (2009) (Preprints in Mathematical Sciences, Lund University, Centre for Mathematical Sciences, Centrum Scientiarum Mathematicarum (2007:25) LUTFMA-5091-2007 and in arXiv:0709.2413 [math.RA] (2007))
\bibitem{MakhSilv:HomDeform}
Makhlouf, A., Silvestrov, S.: Notes on $1$-parameter formal deformations of Hom-associative and Hom-Lie algebras, Forum Math. \textbf{22}(4), 715-739 (2010)
(Preprints in Mathematical Sciences, Lund University, Centre for Mathematical Sciences, Centrum Scientiarum Mathematicarum, (2007:31) LUTFMA-5095-2007. arXiv:0712.3130v1 [math.RA] (2007))
\bibitem{MakhSilv:HomAlgHomCoalg}
Makhlouf, A., Silvestrov, S. D.:
Hom-algebras and Hom-coalgebras, J. Algebra Appl. \textbf{9}(4), 553-589 (2010) (Preprints in Mathematical Sciences, Lund University, Centre for Mathematical Sciences, Centrum Scientiarum Mathematicarum, (2008:19) LUTFMA-5103-2008. arXiv:0811.0400
[math.RA] (2008)) % DOI: 10.1142/S0219498810004117
\bibitem{MandalMishra:HomGerstenhaberHomLiealgebroids}
Mandal, A., Mishra, S. K.: On Hom-Gerstenhaber algebras, and Hom-Lie algebroids, J. Geom. Phys. \textbf{133}, 287-302 (2018)
\bibitem{MishraSilvestrov:SpringerAAS2020HomGerstenhalgsHomLiealgds}
Mishra, S. K., Silvestrov, S.:  A review on Hom-Gerstenhaber algebras and Hom-Lie algebroids, In: Silvestrov S., Malyarenko A., Ran\u{c}i\'{c}, M. (Eds.), Algebraic Structures and Applications, Springer Proceedings in Mathematics and Statistics, \textbf{317}, Springer, Ch. 11, 285-315 (2020)
\bibitem{NovotnyHrivnak:abgderivLiealinvfncs}	
Novotn\'{y}, P., Hrivn\'{a}k, J.: On $(\alpha,\beta,\gamma)$-derivations of Lie algebras and corresponding
invariant functions, J. Geom. Phys. \textbf{58}(2), 208–217 (2008)
\bibitem{PojidaevSaraiva:Derivternmalcevalg}
Pojidaev, A, Saraiva, P.: On derivations of the ternary Malcev algebra M8, Comm. Algebra, \textbf{34}, 3593-3608 (2006)
\bibitem{RichardSilvJA2008:quasiLiesigderCtpm1}
Richard, L., Silvestrov, S. D.: Quasi-Lie structure of $\sigma$-derivations of $\mathbb{C}[t^{\pm1}]$, J. Algebra, \textbf{319}(3), 1285-1304 (2008)
\bibitem{RichardSilvestrovGLTMPBSpr2009:QuasiLieHomLiesigmaderiv}
Richard, L., Silvestrov, S.:
A Note on Quasi-Lie and Hom-Lie structures of $\sigma$-derivations of $\mathbb{C}[z_1^{\pm1},\dots,z_n^{\pm1}]$, In: Silvestrov, S., Paal, E., Abramov, V., Stolin, A. (Eds.), Generalized Lie Theory in Mathematics, Physics and Beyond, Springer-Verlag, Berlin, Heidelberg, Ch. 22, 257-262 (2009)
\bibitem{Saadaoui:ClassmultiplsimpleBiHomLiealg}
Saadaoui, N.: Classification of multiplicative simple BiHom-Lie algebras, arXiv: 1911.09942v2 [math.RA] (2019)
\bibitem{Sheng:homrep}
Sheng, Y.: Representations of Hom-Lie algebras, Algebr. Reprensent. Theory, \textbf{15}, 1081-1098 (2012)
\bibitem{ShengBai2014:homLiebialg}
Sheng, Y., Bai, C.: A new approach to Hom-Lie bialgebras, J. Algebra, \textbf{399}, 232-250 (2014)
\bibitem{ShengChen2013:HomLie2algebras}
Sheng, Y., Chen, D.: Hom-Lie $2$-algebras, J. Algebra \textbf{376}, 174-195 (2013)
\bibitem{ShengXiong:LMLA2015:OnHomLiealg}
Sheng, Y., Xiong Z.: On Hom–Lie algebras, Linear Multilinear Algebra, \textbf{63}(12), 2379–2395 (2015)
\bibitem{SigSilv:CzechJP2006:GradedquasiLiealgWitt}
Sigurdsson, G., Silvestrov, S.: Graded quasi-Lie algebras of Witt type, Czechoslovak J. Phys. \textbf{56}, 1287-1291 (2006)
\bibitem{SigSilv:GLTbdSpringer2009}
Sigurdsson, G., Silvestrov, S.: Lie color and Hom-Lie algebras of Witt type and their central extensions, In: Silvestrov, S., Paal, E., Abramov, V., Stolin, A. (Eds.), Generalized Lie Theory in Mathematics, Physics and Beyond, Springer-Verlag, Berlin, Heidelberg, Ch. 21, 247-255 (2009)
\bibitem{SilvestrovParadigmQLieQhomLie2007}
Silvestrov, S.: Paradigm of quasi-Lie and quasi-Hom-Lie algebras and quasi-defor\-mations, In: New techniques in Hopf algebras and graded ring theory, K. Vlaam. Acad. Belgie Wet. Kunsten (KVAB), Brussels, 165-177 (2007)
\bibitem{Williams:NilpnLiealg}
Williams, M. P.: Nilpotent $n$-Lie algebras, Comm. Algebra, \textbf{37}(6), 1843-1849 (2009) %DOI: 10.1080/00927870802108007
\bibitem{Yau:EnvLieAlg}
Yau, D.: Enveloping algebras of Hom-Lie algebras, J. Gen. Lie Theory Appl. \textbf{2}(2), 95-108 (2008)
\bibitem{Yau2009:HomYangBaxterHomLiequasitring} Yau, D.: Hom-Yang-Baxter equation, Hom-Lie algebras, and quasi-triangular bialgebras, J. Phys. A, \textbf{42}, 165202 (2009)
\bibitem{Yau:HomolHom}
Yau, D.: Hom-algebras and homology, J. Lie Theory, \textbf{19}(2), 409-421 (2009)
\bibitem{Yau:HomBial} Yau, D.: Hom-bialgebras and comodule algebras, Int. Electron. J. Algebra, \textbf{8}, 45-64 (2010)
\bibitem{Yuan2012:HomLiecoloralgstr} Yuan, L.: Hom-Lie color algebra structures, Comm. Algebra, \textbf{40}, 575-592 (2012)
\bibitem{ZhangZhang:GenDerLieSuperalg}
Zhang, R., Zhang, Y.: Generalized derivations of Lie superalgebras, Comm. Algebra, \textbf{38}(10), 3737-3751 (2010)
\bibitem{ZhouChenMa:GenDerHomLiesuper}
Zhou, J., Chen, L., Ma, Y.: Generalized derivations of Hom-Lie superalgebras, Acta Math. Sinica (Chin. Ser.) \textbf{58}, 3737-3751 (2014)
\bibitem{ZhouChenMa:GenDerLieTripSyst}
Zhou, J., Chen, L., Ma, Y.,: Generalized derivations of Lie triple systems,
Open Math., \textbf{14}(1), 260–271 (2016) (arXiv:1412.7804 (2014))
\bibitem{ZhouNiuChen:GhomDerivation}
Zhou, J., Niu, Y. J., Chen, L. Y.: Generalized derivations of Hom-Lie algebras,
Acta Mathematica Sinica, Chinese Series, \textbf{58}(4), 551-558 (2015)
\bibitem{ZhouZhaoZhang:GenDerHomLeibnizalg}
Zhou, J., Zhao, X., Zhang, Y.: Generalized derivations of Hom-Leibniz algebras, J. Jilin University (Science Edition), \textbf{55}(02), 195-200 (2017)	
\end{thebibliography}
\end{document}